\documentclass[10pt]{amsart}

\usepackage[utf8]{inputenc}
\usepackage[T2A]{fontenc}
\usepackage[english]{babel}
\usepackage{amssymb,amsmath}
\usepackage{mathtools}
\usepackage{amsthm}
\usepackage{hyperref}
\usepackage{url}
\usepackage{tikz-cd}
\usepackage[shortlabels]{enumitem}
\usepackage{geometry}
\newtheorem{theorem}{Theorem}[section]
\newtheorem{proposition}[theorem]{Proposition}
\newtheorem{lemma}[theorem]{Lemma}
\newtheorem{corollary}[theorem]{Corollary}

\theoremstyle{definition}
\newtheorem{definition}[theorem]{Definition}
\newtheorem{remark}[theorem]{Remark}

\numberwithin{equation}{section}

\newcommand{\Z}{\mathbb{Z}}
\newcommand{\Q}{\mathbb{Q}} 
\newcommand{\V}{\mathbb{V}}
\newcommand{\Qp}{\mathbb{Q}_p}

\newcommand{\Zhat}{\widehat{\mathbb{Z}}}

\newcommand{\Oc}{\mathcal{O}}

\DeclareMathOperator{\GL}{GL}
\DeclareMathOperator{\Gal}{Gal}
\DeclareMathOperator{\Spec}{Spec}
\DeclareMathOperator{\Hom}{Hom}
\DeclareMathOperator{\Isom}{Isom}
\DeclareMathOperator{\Out}{Out}
\DeclareMathOperator{\rank}{rank}
\DeclareMathOperator{\tr}{tr}

\DeclareMathOperator{\Path}{Path}

\DeclareMathOperator{\Fet}{F\acute{e}t}

\DeclareMathOperator{\Sec}{Sec}
\DeclareMathOperator{\Cusp}{Cusp}
\DeclareMathOperator{\Sel}{Sel}

\title{Families of curves separating points}
\author{Wojciech Porowski}
\address{Research Institute for Mathematical Sciences, Kyoto University, Kyoto 606-8502, Japan}
\email{porowski@kurims.kyoto-u.ac.jp}  

\begin{document}

\begin{abstract}
We show that for a hyperbolic curve $X$ over an algebraically closed field $k$ of characteristic zero and a finite set of points $S\subset X(k)$, after replacing $X$ by a finite \'etale cover, there exists a family of curves over $X$ whose fibres over points in $S$ have pairwise nonisogenous Jacobians.
We also give an application of our results to a problem concerning Galois sections.
\end{abstract}

\maketitle

\section{Introduction}\label{s:introduction}

Let $k$ be an algebraically closed field and $X$ be a hyperbolic curve over $k$.
Consider a family of smooth proper connected curves $Y\to X$, in algebraic geometry we are often interested in the behaviour of a particular geometric property of a fibre $Y_s$ depending on the point $s\in X$.
For example, we could ask how often a particular isogeny class of abelian varieties appears among the Jacobians of $Y_s$ when we vary $s\in X(k)$ along closed points (see \cite{orr_2015families}).
However in the present work we slightly change the perspective by fixing points on the base and varying the family of curves instead.

More precisely, let $S\subset X(k)$ be a finite set of $k$-rational points of $X$, the following definition is the starting point of this paper.

\begin{definition}\label{s1:def:separation_by_isogeny}
We say that the family $Y\to X$ \emph{separates points} in $S$ by \emph{isogeny classes} if 
the Jacobians of the fibres $Y_s$ over points in $s\in S$ are pairwise nonisogenous.
\end{definition}

Then the question we ask is the following: fixing a finite set $S\subset X(k)$ of $k$-rational points, does there exist a family of curves $Y\to X$ which separates points in $S$ by isogeny classes?
Our first theorem states that if we are allowed to replace $X$ be a finite \'etale cover then in characteristic zero the answer to this question is affirmative, in the strong sense.

\begin{theorem}\label{s1:thm:separation_in_char_zero}
Assume that $k$ has characteristic zero.
Then for any finite subset $S\subset X(k)$ there exists a connected finite \'etale cover $X'\to X$ such that   for any lift $S'\subset X'(k)$ of $S$ there exists a family of smooth proper connected curves $Y\to X'$ which separates points in $S'$ by isogeny classes.
\end{theorem}

Before we comment on the proof of Theorem~\ref{s1:thm:separation_in_char_zero} we introduce a more refined version of Definition~\ref{s1:def:separation_by_isogeny}.
Assume that $k$ is an algebraic closure of the fraction field of a strictly henselian discrete valuation ring thus every abelian variety $A$ over $k$ extends uniquely to a semiabelian scheme $\mathcal{A}$ over the ring of integers of $k$.
The special fibre of $\mathcal{A}$ is a semiabelian variety and its toric rank will be referred to simply as the toric rank of $A$.
Then our second definition is as follows.

\begin{definition}\label{s1:def:separation_by_toric_rank}
We say that the family $Y\to X$ \emph{separates points} in $S$ by the \emph{toric rank} if the Jacobians of the fibres $Y_s$ over points in $s\in S$ have pairwise distinct toric rank.
\end{definition}

With this definition we may state our second theorem which is a more precise version of Theorem~\ref{s1:thm:separation_in_char_zero}; note that it applies in particular to the case of an algebraic closure of the field of $p$-adic numbers $\Qp$.

\begin{theorem}\label{s1:thm:separation_by_toric_rank}
Assume that $k$ has characteristic zero and the residue field of $k$ is an algebraic closure of the finite field $\mathbb{F}_p$, for some prime number $p$. Then for any finite subset $S\subset X(k)$ there exists a connected finite \'etale cover $X'\to X$ such that   for any lift $S'\subset X'(k)$ of $S$ there exists a family of smooth proper connected curves $Y\to X'$ which separates points in $S'$ by the toric rank.
\end{theorem}

Our proof of Theorems~\ref{s1:thm:separation_in_char_zero} and~\ref{s1:thm:separation_by_toric_rank} is essentially $p$-adic and uses two ingredients.
The first one is combinatorial anabelian geometry, which is the analysis of the \'etale fundamental group of a hyperbolic curve $X$ using combinatorial data associated to the dual graph of the special fibre of the stable model of $X$; the key concept we use here is the notion of a graphic isomorphism, introduced in \cite{mochizuki_2007}.
The second one is a resolution of nonsingularities (\cite{tamagawa_2004_resolution}, \cite{mochizuki_tsujimura_2023}), which allows one to distinguish points of $X(k)$ by using geometry of the dual graphs of finite \'etale covers of $X$.
We remark here that our proof uses a compactness argument thus does not give an effective bound on the degree of a finite \'etale cover $X'\to X$ appearing in the statements of Theorems~\ref{s1:thm:separation_in_char_zero} and~\ref{s1:thm:separation_by_toric_rank}. 

We also give an application of our results to a problem concerning Galois sections; in fact this application was the main motivation for the author to consider families of curves separating points.
Let $K$ be a number field and $X$ be a hyperbolic curve over $K$; for simplicity we assume that $X$ is proper and has stable reduction over $K$, write $\mathcal{X}$ for the stable model of $X$.
Write $\V(K)$ for the set of nonarchimedean valuations of $K$, for a valuation $v\in \V(K)$ we denote by $K_v$ the completion of $K$ at $v$ and by $\kappa(v)$ the residue field of $v$.
We have a homotopy short exact sequence of \'etale fundamental groups (see \cite[Exp. IX, Thm. 6.1]{sga1})
\[
1\to \pi_1(X\times_K \bar{K})\to \pi_1(X) \to G_K\to 1,
\]
where $\bar{K}$ is an algebraic closure of $K$ and $G_K = \Gal(\bar{K}/K)$ is the absolute Galois group of $K$.
Write $\Sec(X)$ for the set of conjugacy classes of sections of the above sequence, there is a natural injective map $X(K)\hookrightarrow \Sec(X)$.
Given a section $s\in \Sec(X)$ and a valuation $v\in \V(K)$ we denote by $s_v$ the restriction of $s$ to a decomposition group $G_v\subset G_K$.
We say that $s$ is a \emph{Selmer} section (see \cite{porowski2025}) if for every $v\in \V(K)$ the local section $s_v$ is geometric, i.e., it arises from a unique point in $X(K_v)$.
Write $\Sel(X)$ for the set of conjugacy classes of Selmer sections thus we have a sequence of inclusions
\[
X(K)\subset \Sel(X)\subset \Sec(X).
\]
The Section Conjecture proposed by Grothendieck (see~\cite{brief_an_faltings}) predicts that both of these inclusions are actually equalities.
We refer the reader to~\cite{stix2012rational} for more information on the Section Conjecture.

We recall from~\cite{porowski2025} the notion of a reduction of a Selmer section.
For every $v\in \V(K)$, by considering the reductions of points on the stable model $\mathcal{X}$, we obtain a map
\[
red_v\colon \Sel(X)\to X(K_v)\to \mathcal{X}(\kappa(v)),
\]
and for a subset $\Omega\subset \V(K)$ we define the reduction map on $\Omega$
\[
red_{\Omega}\colon \Sel(X) \to \prod_{v\in \Omega} \mathcal{X}(\kappa(v))
\]
as the product of the corresponding maps $red_v$, for all $v\in \Omega$.
It is an important open problem to show that the map $red_{\Omega}$ is injective whenever the set $\Omega$ is infinite; note that this injectivity follows immediately from the Section Conjecture.
It was proved by the author in~\cite{porowski2025} that $red_{\Omega}$ is injective when $\Omega$ has upper density one.
Here we can slightly improve this result; first recall from~\cite{porowski2025} that a subset $\Omega\subset \V(K)$ is said to have \emph{strongly positive density} if for every finite field extension $L/K$ the preimage of $\Omega$ under the restriction map 
\[
\V(L)\to \V(K)
\]
has positive upper density. Then our result is the following theorem.
\begin{theorem}\label{s1:thm:reduction_on_strongly_pos_density}
Assume that the set $\Omega$ has strongly positive density. Then the map $red_{\Omega}$ is injective.
\end{theorem}
In fact our results are somewhat stronger than what is stated in Theorem~\ref{s1:thm:reduction_on_strongly_pos_density} as we can also give examples of infinite sets $\Omega$ of density zero for which the map $red_{\Omega}$ is still injective; we refer the interested reader to Section~\ref{s:galois_sections}.

We now sketch the structure of the present paper.
In Section~\ref{s:families} we discuss homotopy exact sequences coming from families of hyperbolic or, more generally, stable curves; this gives a method of constructing new families of curves from old ones.
In Section~\ref{s:graphic} we discuss the notion of a graphic isomorphism between $p$-prime fundamental groups of hyperbolic curves and we prove a numerical criterion for checking whether an isomorphism is graphic; we extend these results to an arbitrary finite number of isomorphisms in Section~\ref{s:criterion}.
Then in Section~\ref{s:separating_classes} we generalize the problem of separating points into separating \emph{classes} of points, we also introduce various partitions of the set $X(k)$ into classes.
Then we apply our discussion from Section~\ref{s:families} together with results from Section~\ref{s:criterion} to construct families of curves separating classes of points.
In Section~\ref{s:separating_points} we use a resolution of nonsingularities to extend our results from Section~\ref{s:separating_classes} from separating classes to separating points; here we prove Theorems~\ref{s1:thm:separation_in_char_zero} and~\ref{s1:thm:separation_by_toric_rank}.
Finally in Section~\ref{s:galois_sections} we apply our construction from Section~\ref{s:separating_classes} to prove a number of results concerning Selmer sections of hyperbolic curves over number fields, in particular we prove Theorem~\ref{s1:thm:reduction_on_strongly_pos_density}.

\section*{Acknowledgements}

I would like to thank Yuichiro Hoshi for many hours of discussions on combinatorial anabelian geometry and Shinichi Mochizuki for pointing out an improvement of Proposition~\ref{s5:prop:graphical_equivalence_criterion}.

\section{Families of curves}\label{s:families}

In this section we will discuss a homotopy short exact sequence arising from a family of smooth or, more generally, stable curves.
This sequence gives a method of constructing new families of curves from a given family by considering an \'etale cover of a fibre and extending it to an \'etale cover of the total space.
We will also introduce comparison isomorphisms between fundamental groups of different fibres in a family.

We start by fixing some notation and terminology used in this paper.
Write $\mathbb{P}$ for the set of prime numbers and let $\Sigma\subset \mathbb{P}$ be a nonempty subset.
For a profinite group $G$ we write 
\[
G\twoheadrightarrow G^{\Sigma}
\]
for the maximal pro-$\Sigma$ quotient of $G$.
When $p$ is a prime number we say that $G$ is a $p$\emph{-prime} group if every open subgroup of $G$ has index prime to $p$.
For an open subgroup $H\subset G$ we say that $H$ is a \emph{$\Sigma$-subgroup} if the open normal subgroup 
\[
N = \cap_{g\in G}gHg^{-1}\subset G
\]
has index divisible only by primes belonging to $\Sigma$.

For a scheme $X$ and a point $x\in X$ we write $\bar{x}$ for a geometric point of $X$ lying over $x$.
When $X$ is connected, we write $\pi_1(X,\bar{x})$ for its \'etale fundamental group with the basepoint at the geometric point $\bar{x}$.
We will omit a basepoint from the notation whenever it does not lead to confusion.
For a nonempty set of prime numbers $\Sigma$ the quotient
\[
\pi_1(X)\twoheadrightarrow \pi^{\Sigma}_1(X)
\]
will be called the pro-$\Sigma$ fundamental group of $X$,
in the special case when $\Sigma = \mathbb{P}\setminus\{p\}$ for some prime number $p$ we will refer to $\pi^{\Sigma}_1(X)$ as the $p$-\emph{prime} fundamental group of $X$.
When $X' \to X$ is a connected finite \'etale cover we say that $X'$ is a \emph{$\Sigma$-cover} if the degree of the Galois closure of $X'$ over $X$ is divisible only by primes contained in $\Sigma$.
Thus $\Sigma$-covers of $X$ correspond precisely to $\Sigma$-subgroups of $\pi_1(X)$ and the group $\pi^{\Sigma}_1(X)$ classifies $\Sigma$-covers of $X$.

Let $S$ be a connected and nonempty scheme.
Consider a morphism of schemes 
\begin{equation}\label{s2:eq:compactification}
\overline{X}\to S
\end{equation}
which is proper, flat, of finite presentation, with geometric fibres purely of dimension one and a closed immersion 
\begin{equation}\label{s2:eq:divisor}
D\hookrightarrow \overline{X}
\end{equation}
such that the composition $D\to S$ is finite \'etale of constant degree $r\ge 0$.
We refer to this type of data as a \emph{family of curves over} $S$ and denote it by $\mathcal{F}$,
also we write $X\subset \overline{X}$ for an open subscheme obtained as the complement of $D$.
Note that the arithmetic genus of a geometric fibre of the morphism~\eqref{s2:eq:compactification} is constant on $S$, we denote it by $g\ge 0$.

\begin{definition}\label{s2:def:family_of_curves}
We say that $\mathcal{F}$ is a family of \emph{smooth} (resp. \emph{proper}, \emph{connected}) curves if the morphism~\eqref{s2:eq:compactification} is smooth (resp. $r=0$, the geometric fibres of~\eqref{s2:eq:compactification} are connected).
\end{definition}

For a family of smooth curves $\mathcal{F}$ the scheme $X$ is connected if and only if $\overline{X}$ is connected, additionally connectivity of the geometric fibres of~\eqref{s2:eq:compactification} implies that $\overline{X}$ is connected.
For a family of smooth connected curves $\mathcal{F}$ we say that $\mathcal{F}$ is a family of \emph{hyperbolic} curves if $2g-2+r> 0$.
Suppose that $S$ is a normal scheme, then for a family of hyperbolic curves $\mathcal{F}$ the $S$-scheme $\overline{X}$ and the relative divisor $D\hookrightarrow\overline{X}$ are uniquely determined by the $S$-scheme $X$.
In such a situation we will say simply that $X\to S$ is a family of hyperbolic curves and we will refer to $\overline{X}\to S$ as the \emph{compactification} of the family $X\to S$ and to $D\hookrightarrow\overline{X}$ as the \emph{divisor of cusps}.

From now on we assume that the structure morphism $S\to \Spec \Z$ is not surjective and we fix a nonempty set of prime numbers $\Sigma\subset\mathbb{P}$ which are invertible on $S$.
Consider a family $\mathcal{F}$ of smooth connected curves, we will recall a homotopy exact sequence associated to this family.
From our discussion we see that the natural homomorphism 
\[
\pi_1(X)\twoheadrightarrow \pi_1(S)
\]
is surjective, write $K$ for its kernel so that we have a short exact sequence
\[
1\to K \to \pi_1(X)\to \pi_1(S)\to 1.
\]
Note that the kernel of the quotient $K\twoheadrightarrow K^{\Sigma}$ is a characteristic subgroup of $K$ hence a normal subgroup of $\pi_1(X)$; denote by 
\[
\pi_1(X)\twoheadrightarrow \pi'_1(X)
\]
the corresponding quotient of $\pi_1(X)$.
Thus we have a short exact sequence
\[
1\to K^{\Sigma}\to \pi'_1(X) \to \pi_1(S)\to 1.
\]
Let $\bar{s}$ be a geometric point of $S$, then the natural map $\pi_1(X_{\bar{s}})\to \pi_1(X)$ factorizes through $K\subset \pi_1(X)$, moreover by~\cite[XIII, Proposition 4.1]{sga1} the induced homomorphism of pro-$\Sigma$ quotients
\[
\pi^{\Sigma}_1(X_{\bar{s}})\twoheadrightarrow K^{\Sigma}
\]
is surjective.
Thus we have an exact sequence of profinite groups.
\begin{equation}\label{s2:eq:homotopy_sequence}
\pi^{\Sigma}_1(X_{\bar{s}}) \to \pi'_1(X)\to \pi_1(S)\to 1.
\end{equation}
Finally, when the family $\mathcal{F}$ is a family of hyperbolic curves then by \cite[Proposition 2.7]{stix_2005} the above sequence is actually short exact
\begin{equation}\label{s2:ses:homotopy_ses}
1\to \pi^{\Sigma}_1(X_{\bar{s}}) \to \pi'_1(X)\to \pi_1(S)\to 1.
\end{equation}
In particular, we have a natural outer action
\begin{equation}\label{s2:eq:outer_monodromy_smooth}
\pi_1(S,\bar{s})\to \Out(\pi^{\Sigma}_1(X_{\bar{s}})).
\end{equation}

Before we continue we briefly explain how to use the sequence~\eqref{s2:ses:homotopy_ses} to construct families of hyperbolic curves by taking \'etale covers of a given family and passing to compactifications.
Suppose that $S$ is regular and let $X\to S$ be a family of hyperbolic curves, $\bar{s}$ be a geometric point of $S$ and
\[
U_{\bar{s}}\subset \pi^{\Sigma}_1(X_{\bar{s}})
\]
be an open subgroup.
From the left exactness of the sequence~\eqref{s2:ses:homotopy_ses} it follows that there exists an open subgroup $U\subset \pi'_1(X)$ such that
\[
U\cap \pi^{\Sigma}_1(X_{\bar{s}}) = U_{\bar{s}}.
\]
The open subgroup $U\subset \pi'_1(X)$ corresponds to a connected finite \'etale cover $X''\to X$.
Write 
\[
V\subset \pi_1(S)
\]
for the image of $U$ in $\pi_1(S)$, it is an open subgroup corresponding to a connected finite \'etale cover $S'\to S$.
Thus we obtain a commutative diagram with a cartesian square
\begin{equation}\label{s2:diag:construction_of_a_family}
\begin{tikzcd}
X''\arrow[r]\arrow[rd]& X'\arrow[r] \arrow[d]& X \arrow[d] \\
 & S'\arrow[r] & S,
\end{tikzcd}
\end{equation}
whose horizontal arrows are finite \'etale morphisms.
Clearly $S'$ is regular and $X'\to S'$ is a family of hyperbolic curves, with compactification $\overline{X}'$ and the divisor of cusps $D'$.
Since $X''\to X'$ is a $\Sigma$-cover and $D'$ is a regular divisor on a regular scheme $\overline{X}'$ it follows from Abhyankar's lemma that $X''\to X'$ extends uniquely to a tamely ramified cover 
\[
\overline{X}''\to \overline{X}'.
\]
Moreover the preimage $D''\subset \overline{X}''$ of the divisor $D'$, considered with reduced scheme structure, is \'etale over $S'$.
By construction $U$ surjects onto $\pi_1(S')$ thus by considering Stein factorization we see that the geometric fibres of the morphism
\[
\overline{X}''\to S'
\]
are connected.
Therefore we deduce that $X''\to S'$ is a family of hyperbolic curves.
Summarizing the above discussion, starting from a family of hyperbolic curves $X\to S$, a geometric point $\bar{s}$ of $S$ and a connected finite \'etale cover of the fibre $X_{\bar{s}}$ we have constructed a new family $X''\to S'$ over a connected finite \'etale cover of $S$; in Section~\ref{s:separating_classes} we will employ this procedure to construct desired families of curves.

Our next goal is to relate fundamental groups of different fibres in a family of hyperbolic curves.
Recall that for any two geometric points $\bar{s}$ and $\bar{t}$ of a connected scheme $S$ we have a set of paths
\[
\Path(\bar{s},\bar{t})
\]
consisting of isomorphisms of profinite groups
\[
\pi_1(S,\bar{s}) \xrightarrow{\sim} \pi_1(S,\bar{t})
\]
coming from isomorphisms between corresponding fibre functors.
The set $\Path(\bar{s},\bar{t})$ has a natural transitive left action by $\pi_1(S,\bar{t})$ as well as a transitive right action by $\pi_1(S,\bar{s})$; there are also natural composition and inversion operations on paths.
For a connected $S$-scheme $S'$ with two geometric points $\bar{s}'$ and $\bar{t}'$, lying over $\bar{s}$ and $\bar{t}$ respectively, there is a natural map of sets
\[
\Path(\bar{s}',\bar{t}')\to \Path(\bar{s},\bar{t})
\]
which is compatible with the corresponding left and right actions of the fundamental groups of $S'$ and $S$ via the homomorphism $\pi_1(S')\to \pi_1(S)$.

Let $\mathcal{F}$ be a family of hyperbolic curves over $S$ and $\bar{s}$ be a geometric point of $S$.
Choosing a geometric point $\bar{a}$ of $X_{\bar{s}}$ we may write the sequence~\eqref{s2:ses:homotopy_ses} using basepoints as
\begin{equation}\label{s2:ses:homotopy_ses_basepoints}
1\to \pi^{\Sigma}_1(X_{\bar{s}}, \bar{a}) \to \pi'_1(X,\bar{a})\to \pi_1(S, \bar{s})\to 1.
\end{equation}
Let $\bar{t}$ be another geometric point of $S$ and $\bar{b}$ be a geometric point of $X_{\bar{t}}$.
Then for a path isomorphism 
\[
p_X\colon \pi'_1(X,\bar{a})\xrightarrow{\sim} \pi'_1(X,\bar{b}),
\]
there exists a unique path isomorphism
\[
p_S\colon \pi_1(S,\bar{s})\xrightarrow{\sim} \pi_1(S,\bar{t})
\]
making the following diagram commutative
\[
\begin{tikzcd}
1\arrow[r] & \pi_1^{\Sigma}(X_{\bar{s}},\bar{a}) \arrow[r]\arrow[d, "\sim" sloped] & \pi'_1(X, \bar{a})\arrow[r]\arrow[d, "p_X" ] & \pi_1(S, \bar{s})\arrow[r]\arrow[d, "p_S"] & 1 \\
1\arrow[r] & \pi_1^{\Sigma}(X_{\bar{t}},\bar{b}) \arrow[r] & \pi'_1(X, \bar{b})\arrow[r] & \pi_1(S, \bar{t})\arrow[r] & 1.
\end{tikzcd}
\]
Denote by
\[
\alpha \in \Isom^{\Out} (\pi^{\Sigma}_1(X_{\bar{s}}), \pi^{\Sigma}_1(X_{\bar{t}}))
\]
the outer isomorphism determined by the left vertical arrow in the above diagram.
Note that $\alpha$ does not depend on the choice of basepoints $\bar{a}$ and $\bar{b}$ of the fibres $X_{\bar{s}}$ and $X_{\bar{t}}$.

\begin{definition}\label{s2:def:comparison_smooth}
Any outer isomorphism $\alpha$ obtained from a path $p_X$ as above will be called a \emph{comparison isomorphism} between fibres over $\bar{s}$ and $\bar{t}$; we denote by
\[
C(\mathcal{F},\bar{s},\bar{t})\subset \Isom^{\Out} (\pi^{\Sigma}_1(X_{\bar{s}}), \pi^{\Sigma}_1(X_{\bar{t}}))
\]
the set of all comparison isomorphisms.
\end{definition}

From the above definition we see that the set $C(\mathcal{F},\bar{s},\bar{t})$ has two transitive group actions, namely a left action by the group $\pi_1(S,\bar{s})$ and a right action by the group $\pi_1(S,\bar{s})$.
Moreover the operation of path composition induces naturally the operation of composition of comparison isomorphisms between fibres.
In the case $\bar{s} = \bar{t}$ the set of comparison isomorphism is simply the image of the outer monodromy representation~\eqref{s2:eq:outer_monodromy_smooth}.

Even though we are mainly interested in families of smooth curves, it will be necessary to consider more general families of stable curves.

\begin{definition}
Let $\mathcal{F}$ be a family of curves over $S$ we say that $\mathcal{F}$ is a family of \emph{semistable} curves if the geometric fibres of the morphism~\eqref{s2:eq:compactification} have only nodes as singularities and moreover $D$ lies in the smooth locus of~\eqref{s2:eq:compactification}.
Additionally, if every geometric fibre of $\mathcal{F}$ is a stable curve then we say that $\mathcal{F}$ is a family of \emph{stable} curves.
\end{definition}

When working with families of semistable curves it will be convenient to use the language of log geometry and log fundamental groups (see~\cite{illusie_overview}).
For a log scheme $S^{log}$ we write $S$ for the underlying scheme, when $S$ is connected then for a log geometric point $\tilde{s}$ of $S^{log}$ we write $\pi_1(S^{log},\tilde{s})$ for the log fundamental group of $S^{log}$ classifying Kummer \'etale (k\'et) covers of $S^{log}$.
Whenever the choice of a basepoint is irrelevant we will omit it from the notation.
In the following all log schemes we consider are fs log schemes.

Let $S^{log}$ be an fs log scheme and
\begin{equation}\label{s2:eq:log_morphism}
\overline{X}{}^{log}\to S^{log}
\end{equation}
be a morphism of fs log schemes.
The notion of a family of semistable curves can be also defined using log structures, as follows.

\begin{definition}
Morphism~\eqref{s2:eq:log_morphism} is called a \textit{semistable log curve} if it is log smooth, saturated, with connected geometric fibres and purely of relative dimension one.
\end{definition}

See~\cite{tsuji_saturated} for the definition and properties of saturated morphisms of log schemes.
Any semistable log curve~\eqref{s2:eq:log_morphism} determines a closed subscheme $D\subset\overline{X}$ such the underlying morphisms of schemes determine a family of semistable curves over $S$.
Conversely, for any family $\mathcal{F}$ of semistable curves over $S$ there exist log structures on $S$ and $\overline{X}$ which make the morphism $\overline{X}\to S$ into a semistable log curve (see~\cite{kato_fumiharu_2000}).
In particular, the geometric fibres of a semistable log curve are (marked) semistable curves; we say that a semistable log curve~\eqref{s2:eq:log_morphism} is a \emph{stable log curve} if its geometric fibres are (marked) stable curves.

Given a stable log curve~\eqref{s2:eq:log_morphism}, from the above discussion we see that the underlying morphism of schemes determines a family of stable curves thus we obtain a classifying map
\begin{equation}\label{s2:eq:classifying_map}
S\to \overline{\mathcal{M}}_{g,r}
\end{equation}
from $S$ to the stack of pointed stable curves for some $g,r\ge 0$.
The stack $\overline{\mathcal{M}}_{g,r}$ as well as the universal curve $\overline{\mathcal{C}}_{g,r}$ over it have natural fs log structures (see \cite{kato_fumiharu_2000}) for which the morphism
\[
\overline{\mathcal{C}}^{log}_{g,r}\to \overline{\mathcal{M}}^{log}_{g,r}
\]
is a log smooth morphism of log stacks. 
The log structure on $\overline{\mathcal{M}}_{g,r}$ is defined by the normal crossing divisor corresponding to singular curves, in particular the resulting log structure is log regular.
Then the morphism~\eqref{s2:eq:classifying_map} extends to a morphism of log stacks for which the commutative diagram 
\[
\begin{tikzcd}
\overline{X}{}^{log}\arrow[r]\arrow[d] & \overline{\mathcal{C}}_{g,r}^{log}\arrow[d] \\
S^{log}\arrow[r] & \overline{\mathcal{M}}_{g,r}^{log}
\end{tikzcd}
\]
of log stacks is cartesian.

Next we are going to discuss a homotopy exact sequence of log fundamental group associated to a family of semistable curves. Let 
\begin{equation}\label{s2:eq:semistable_log_curve_1}
\overline{X}{}^{log}\to S^{log}
\end{equation}
a semistable log curve, since the morphism~\eqref{s2:eq:semistable_log_curve_1} is saturated it follows that for every morphism
\[
T^{log}\to S^{log}
\]
of fs log schemes the underlying scheme of the fibre product in the category of fs log schemes
\[
Y^{log} = \overline{X}{}^{log}\times_{S^{log}} T^{log}
\]
coincides with the fibre product of $X$ and $T$ over $S$ in the category of schemes.
In particular, $Y$ is connected whenever $T$ is connected thus the natural map between log fundamental groups
\begin{equation}\label{s2:eq:pi_morphism_X_to_S}
\pi_1(\overline{X}{}^{log}) \twoheadrightarrow \pi_1(S^{log})
\end{equation}
is a surjection.

Suppose for a moment that $S$ is the spectrum of a separably closed field $k$ so $X$ is simply a semistable curve over $k$.
We denote by $\pi^{adm}_1(X)$ the kernel of the surjection~\eqref{s2:eq:pi_morphism_X_to_S} thus we have a short exact sequence
\[
1\to \pi^{adm}_1(X)\to \pi_1(\overline{X}{}^{log})\to \pi_1(S^{log})\to 1.
\]
The group $\pi^{adm}_1(X)$ will be referred to as the \emph{admissible} fundamental group of $X$, it does not depend on the log structure on $S$ and can be defined as a profinite group classifying admissible covers of $X$, see~\cite{yang_2018} or \cite{wewers_1999} for the definition of admissible covers and its properties.
In particular, if $p$ is the characteristic of $\bar{s}$ then the maximal pro-$\Sigma$ quotient $\pi^{adm,\Sigma}_1(X)$ can be defined as a fundamental group of a certain graph of groups; moreover its isomorphism class depends only on the arithmetic genus and the number of cusps of $X$ (see \cite{serre1980trees} and \cite[Appendix]{mochizuki_2004}).

We now come back to the case of a semistable log curve~\eqref{s2:eq:semistable_log_curve_1} over a general connected log scheme $S^{log}$.
Observe that for every geometric point $\bar{s}$ of $S$ we have natural maps
\begin{equation}\label{s2:loc:maps}
\pi_1^{adm}(X_{\bar{s}}) \to \pi_1(\overline{X}{}^{log})\twoheadrightarrow \pi_1(S^{log})
\end{equation}
whose composition is trivial.
In the next lemma we show that this gives a homotopy exact sequence associated to a semistable log curve, see \cite{hoshi_exactness_2009} for another variant with different assumptions.

\begin{lemma}\label{s2:lem:right_exact_sequence}
Morphisms~\eqref{s2:loc:maps} induce an exact sequence
\[
\pi^{adm}_1(X_{\bar{s}})\to \pi_1(\overline{X}{}^{log})\to \pi_1(S^{log})\to 1.
\]
\end{lemma}

\begin{proof}
Let $H$ be an open subgroup 
\[
H\subset \pi_1(\overline{X}{}^{log})
\]
corresponding to a k\'et cover 
\begin{equation}\label{s2:loc:Y_log_to_X_log}
Y^{log}\to \overline{X}{}^{log}
\end{equation}
whose restriction over $X_{\bar{s}}$ contains a connected component mapping isomorphically onto $X_{\bar{s}}$.
Then the composition
\begin{equation}\label{s2:loc:Y_log_to_S_log}
Y^{log}\to S^{log}
\end{equation}
is a log smooth exact morphism of fs log schemes.
We need to show that there exists a unique connected k\'et cover of $S^{log}$ whose base change along~\eqref{s2:eq:semistable_log_curve_1} is isomorphic to $Y^{log}$.

Define a subset $S'\subset S$ consisting of all points $s\in S$ such that the restriction of~\eqref{s2:loc:Y_log_to_X_log} to $X_{\bar{s}}$ contains a connected component mapping isomorphically onto $X_{\bar{s}}$.
By assumption $S'$ is nonempty, write $Z$ for the closure of $S'$ in $S$.
We are going to prove the following assertion: for every point $s\in Z$ there exists an open neighbourhood $s\in U\subset S$ and a (unique) k\'et cover 
\begin{equation}\label{s2:loc:V_log_to_U_log}
V^{log}\to U^{log}
\end{equation}
such that the restriction of \eqref{s2:loc:Y_log_to_S_log} to $U$ is isomorphic to the base change of~\eqref{s2:eq:semistable_log_curve_1} along 
\[
V^{log}\to U^{log}\hookrightarrow S^{log}
\]
Observe that this assertion will immediately finish the proof; indeed this will show that $S'\subset S$ is open and closed thus $S' = S$, moreover it follows from the uniqueness that local k\'et covers~\eqref{s2:loc:V_log_to_U_log} glue to a global one whose base change along~\eqref{s2:eq:semistable_log_curve_1} is isomorphic to $Y^{log}$.
Fix $s\in Z$, note that the assertion is a local statement around $s\in S$ thus we may assume that $S$ is a Noetherian scheme.
Further replacing $S^{log}$ by a connected k\'et cover we may assume that $H$ surjects onto $\pi_1(S^{log})$, thus we have to show that~\eqref{s2:loc:Y_log_to_X_log} is an isomorphism.

To prove it we are free to replace $S^{log}$ by a k\'et cover thus by applying~\cite[Theorem 4.9.1]{ogus_lectures} we may assume that~\eqref{s2:loc:Y_log_to_S_log} is saturated.
In particular it follows from \cite[Theorem II.4.2]{tsuji_saturated} that the morphism $Y\to S$ is flat with geometrically reduced fibres.
Thus, considering the Stein factorization of $Y\to S$ we deduce that the geometric fibres of $Y\to S$ are connected.
Finally, this implies that for a point $s'\in S'$ the k\'et cover~\eqref{s2:loc:Y_log_to_X_log} restricts to an isomorphism over $X_{\bar{s}'}$ hence it must be an isomorphism.
\end{proof}

We would like to extend the construction of the short exact sequence~\eqref{s2:ses:homotopy_ses} to the case of a family of stable curves. 
Write $K$ for the kernel of the homomorphism~\eqref{s2:eq:pi_morphism_X_to_S} thus we have a short exact sequence
\[
1\to K\to \pi_1(\overline{X}{}^{log})\to \pi_1(S^{log})\to 1.
\]
By pushing out through the characteristic quotient $K\twoheadrightarrow K^{\Sigma}$ we obtain a short exact sequence
\[
1\to K^{\Sigma}\to \pi'_1(\overline{X}{}^{log}) \to \pi_1(S^{log})\to 1,
\]
thus by Lemma~\ref{s2:lem:right_exact_sequence} we have an exact sequence
\begin{equation}\label{s2:eq:right_exact_sequence_sigma}
\pi^{adm,\Sigma}_1(X_{\bar{s}})\to \pi'_1(\overline{X}{}^{log})\to \pi_1(S^{log})\to 1.
\end{equation}
We are interested in situations when the above sequence becomes exact on the left.
First we recall the case of a strictly henselian local ring.

\begin{lemma}\label{s2:lem:strictly_henselian}
Suppose that $S$ is the spectrum of a strictly henselian local ring, $\bar{s}$ is a geometric point of $S$ and $\bar{t}$ is a geometric point over the closed point of $S$.
\begin{enumerate}[(i)]
\item The sequence~\eqref{s2:eq:right_exact_sequence_sigma} is exact on the left, i.e., we have a short exact sequence
\[
1\to \pi^{adm,\Sigma}_1(X_{\bar{s}})\to \pi'_1(\overline{X}{}^{log})\to \pi_1(S^{log})\to 1.
\]
\item A specialization morphism (see \cite[\S 1.3]{lepage_covers_I_2013})
\[
\pi^{adm,\Sigma}_1(X_{\bar{s}}) \xrightarrow{\sim} \pi^{adm,\Sigma}_1(X_{\bar{t}})
\]
is an isomorphism compatible with outer inclusions into $\pi'_1(\overline{X}{}^{log})$.
\end{enumerate}
\end{lemma}

\begin{proof}
The statement $(i)$ for $\bar{s} = \bar{t}$ follows from~\cite[Proposition 5.1]{olsson_2022} or  \cite[Theorem 1.8]{lepage_covers_I_2013}.
Thus we deduce from the sequence~\eqref{s2:eq:right_exact_sequence_sigma} that the specialization morphism in $(ii)$ is surjective.
As the admissible fundamental group $\pi^{adm,\Sigma}_1(X_{\bar{s}})$ is topologically finitely generated and its isomorphism class is independent of the geometric point $\bar{s}$ of $S$ it follows that the specialization map in $(ii)$ is an isomorphism.
Thus $(i)$ holds for any geometric point $\bar{s}$ of $S$. 
\end{proof}

Consider now a stable log curve
\begin{equation}\label{s2:eq:stable_log_curve}
\overline{X}{}^{log} \to S^{log},
\end{equation}
similarly as in the smooth case we will show that for families of stable curves we obtain a homotopy short exact sequence over a general base.

\begin{proposition}\label{s2:prop:ses_stable}
For the stable log curve~\eqref{s2:eq:stable_log_curve} the sequence~\eqref{s2:eq:right_exact_sequence_sigma} is exact on the left, i.e., we have a short exact sequence
\[
1\to \pi^{adm,\Sigma}_1(X_{\bar{s}})\to \pi'_1(\overline{X}{}^{log})\to \pi_1(S^{log})\to 1.
\]
In particular, for every log geometric point $\tilde{s}$ of $S^{log}$ we have the induced outer monodromy action
\[
\pi_1(S^{log}, \tilde{s}) \to \Out (\pi^{adm,\Sigma}_1(X_{\bar{s}})).
\]
\end{proposition}
\begin{proof}
To prove the statement we need to show that the homomorphism 
\begin{equation}\label{s2:loc:pi_morphism_fibre_to_X}
\pi^{adm,\Sigma}_1(X_{\bar{s}})\to \pi'_1(\overline{X}{}^{log})
\end{equation}
is injective, using specialization isomorphisms it is enough to show the injectivity for only one geometric point $\bar{s}$ of $S$.

Suppose first that $S^{log}$ is a log regular log scheme, thus $\overline{X}{}^{log}$ is log regular as well.
Let $\eta$ be the generic point of $S$ and $\bar{\eta}$ be a geometric point over $\eta$.
We will prove the injectivity of~\eqref{s2:loc:pi_morphism_fibre_to_X} for $\bar{s} = \bar{\eta}$.
Write 
\[
U_S\subset S, \; U_X\subset \overline{X}
\]
for the corresponding open subschemes where the log structure is trivial.
By assumption $U_S$ and $U_X$ are regular nonempty schemes and the map $U_X\to U_S$ is a family of hyperbolic curves.
Thus $\eta\in U_S$ and the geometric generic fibre $X_{\bar{\eta}}$ is a hyperbolic curve.
Consider a commutative diagram
\[
\begin{tikzcd}
& & \pi_1^{\Sigma}(X_{\bar{\eta}})\arrow[d, hook]\arrow[rd, "\varphi"] & & \\
1\arrow[r] & N\arrow[r] & \pi'_1(X_{\eta})\arrow[r] & \pi'_1(\overline{X}^{log})\arrow[r] & 1,
\end{tikzcd}
\]
with the injective vertical arrow and the group $N$ defined to make the row exact.
Our goal is to show that the map $\varphi$ is injective.
Define a normal subgroup
\[
Z \subset \pi'_1(X_{\eta})
\]
as the centralizer of $\pi_1^{\Sigma}(X_{\bar{\eta}})$ in $\pi'_1(X_{\eta})$.
Since $X_{\bar{\eta}}$ is a hyperbolic curve the fundamental group $\pi_1^{\Sigma}(X_{\bar{\eta}})$ has trivial centre thus we see that the intersection 
\[
\pi_1^{\Sigma}(X_{\bar{\eta}})\cap Z
\]
is trivial.
Therefore the injectivity of $\varphi$ is equivalent to the containment $N\subset Z$.
We prove it by introducing a collection of subgroups $N_s$ of $N$ whose conjugates generate $N$ and showing that each of the groups $N_s$ is contained in $Z$.

Let $s\in S$ be a point of height one and $T$ be the spectrum of a strict henselization of the local ring $\Oc_{S,s}$.
Let $\overline{X}_T\to T$ be the base change of $\overline{X}\to S$ via $T\to S$.
Write $\eta_T$ for the generic point of $T$ and $\bar{\eta}_T$ for a geometric generic point over $\eta$.
Then we have a commutative diagram
\[
\begin{tikzcd}
&  & \pi_1^{\Sigma}(X_{\bar{\eta}_T})\arrow[rd, hook, "\varphi_s"]\arrow[d, hook] & & \\
1\arrow[r] & N_s\arrow[r] & \pi'_1(X_{\eta_T})\arrow[r] & \pi'_1(\overline{X}{}^{log}_T) \arrow[r] & 1
\end{tikzcd}
\]
with the map $\varphi_s$ being injective by Lemma~\ref{s2:lem:strictly_henselian} and the group $N_s$ defined to make the row exact.
Choosing an appropriate path isomorphism we also have a commutative square of outer homomorphisms
\[
\begin{tikzcd}
\pi_1^{\Sigma}(X_{\bar{\eta}_T}) \arrow[r, "\sim"]\arrow[d, hook] &  \pi_1^{\Sigma}(X_{\bar{\eta}})\arrow[d, hook] \\
\pi'_1(X_{\eta_T})\arrow[r, hook] & \pi'_1(X_{\eta}).
\end{tikzcd}
\]
Hence, considering $N_s$ as a subgroup of $\pi'_1(X_{\eta})$ defined up to conjugation, we deduce that $N_s$ is contained in both $N$ and $Z$.

Finally, we claim that $N$ is generated by the conjugates of subgroups $N_s$ for all points $s\in S$ of height one.
Indeed from the log purity theorem of Fujiwara-Kato we have an isomorphism
\[
\pi^t_1(U_X) \xrightarrow{\sim} \pi_1(\overline{X}{}^{log})
\]
where $\pi^t_1(U_X)$ is the tame fundamental group of $\overline{X}$ classifying finite \'etale covers of $U_X$ which are tamely ramified at the generic points of the complementary divisor $\overline{X}\setminus U_X$.
On the other hand, by the Zariski-Nagata purity theorem the kernel of the quotient
\[
\pi'_1(X_{\eta})\twoheadrightarrow \pi'_1(U_X)
\]
is generated by the conjugates of inertia subgroups of points of $U_X$ of height one.
Therefore the subgroup $N$, i.e., the kernel of the quotient
\[
\pi'_1(X_{\eta})\twoheadrightarrow \pi'_1(\overline{X}{}^{log})
\]
is also generated by the inertia groups of points of $\overline{X}$ of height one. 
By construction these inertia groups are contained in conjugates of subgroups $N_s$, proving our claim.
This establishes the injectivity of~\eqref{s2:loc:pi_morphism_fibre_to_X} in the log regular case.

Finally we treat the case of a general stable log curve~\eqref{s2:eq:stable_log_curve}.
Consider a cartesian diagram of fs log schemes
\[
\begin{tikzcd}
\overline{X}{}'{}^{log}\arrow[r]\arrow[d] & \overline{X}{}^{log}\arrow[d] \\
S'^{log}\arrow[r] & S^{log},
\end{tikzcd}
\]
thus the morphism
\begin{equation}\label{s2:loc:stable_log_curve'}
\overline{X}{}'{}^{log} \to S'^{log}
\end{equation}
is also a stable log curve.
For a geometric point $\bar{s}'$ of $S'$ lying over $\bar{s}$ we obtain an induced commutative diagram of outer homomorphisms between fundamental groups
\[
\begin{tikzcd}
\pi^{adm,\Sigma}_1(X_{\bar{s}'}) \arrow[r, "\sim"]\arrow[d, "a"] & \pi^{adm,\Sigma}_1(X_{\bar{s}})\arrow[d, "b"] \\
\pi'_1(\overline{X}{}'{}^{log}) \arrow[r, "c"] & \pi'_1(\overline{X}{}^{log}).
\end{tikzcd}
\]
Suppose that the statement holds for the stable log curve~\eqref{s2:eq:stable_log_curve}, clearly the injectivity of $b$ implies that $a$ is injective therefore it holds for~\eqref{s2:loc:stable_log_curve'} as well.
Suppose now that $S'\to S$ is a k\'et cover and the statement holds for the stable log curve~\eqref{s2:loc:stable_log_curve'}, then $a$ and $c$ are injective thus $b$ is injective hence it holds for~\eqref{s2:eq:stable_log_curve} as well.
Therefore, using the fact that log stacks of pointed stable curves are log regular we immediately reduce the proof to the log regular case which we have already considered.
\end{proof}

Next, we extend our definition of comparison isomorphisms from the smooth case to the stable case.
Let $\tilde{a}$ be a log geometric point of $\overline{X}{}^{log}$ lying over a log geometric point $\tilde{s}$ of $S^{log}$.
Then by Proposition~\ref{s2:prop:ses_stable} we have a short exact sequence 
\[
1\to   \pi^{adm,\Sigma}_1(X_{\bar{s}}, \tilde{a})  \to \pi'_1(\overline{X}{}^{log}, \tilde{a}) \to \pi_1(S^{log},\tilde{s})\to 1.
\]
Let $\tilde{t}$ be another log geometric point of $S^{log}$ and $\tilde{b}$ be a log geometric point of $\overline{X}{}^{log}$ lying over $\tilde{t}$.
Then for a log path
\[
p^{log}_X\colon \pi'_1(\overline{X}{}^{log},\tilde{a})\xrightarrow{\sim} \pi'_1(\overline{X}{}^{log},\tilde{b})
\]
there exists a unique log path
\[
p^{log}_S\colon \pi_1(S^{log},\tilde{s})\xrightarrow{\sim} \pi_1(S^{log},\tilde{t})
\]
making the following diagram commutative
\begin{equation}\label{s2:diag:log_path_diagram}
\begin{tikzcd}
1\arrow[r] & \pi_1^{adm, \Sigma}(X_{\bar{s}},\tilde{a}) \arrow[r]\arrow[d, "\sim" sloped] & \pi'_1(\overline{X}{}^{log}, \tilde{a})\arrow[r]\arrow[d, "p^{log}_X"] & \pi_1(S^{log}, \tilde{s})\arrow[r]\arrow[d, "p^{log}_S"] & 1 \\
1\arrow[r] & \pi_1^{adm, \Sigma}(X_{\bar{t}},\tilde{b}) \arrow[r] & \pi'_1(\overline{X}{}^{log}, \tilde{b})\arrow[r] & \pi_1(S^{log}, \tilde{t})\arrow[r] & 1.
\end{tikzcd}
\end{equation}
Denote by
\[
\alpha \in \Isom^{\Out} (\pi^{adm,\Sigma}_1(X_{\bar{s}}), \pi^{adm,\Sigma}_1(X_{\bar{t}}))
\]
the outer isomorphism determined by the left vertical arrow in the above diagram.

\begin{definition}
Any outer isomorphism $\alpha$ obtained from a log path $p^{log}_X$ as above will be called a \emph{ comparison isomorphism} between fibres over $\bar{s}$ and $\bar{t}$; we denote by
\[
C(\overline{X}{}^{log}/S^{log},\bar{s},\bar{t})\subset \Isom^{\Out} (\pi^{adm,\Sigma}_1(X_{\bar{s}}), \pi^{adm,\Sigma}_1(X_{\bar{t}}))
\]
the set of all comparison isomorphisms.
\end{definition}

Clearly this extends Definition~\ref{s2:def:comparison_smooth} to the stable case.
Similarly as in the smooth case there are two transitive group actions on the set of comparison isomorphisms: a left action by $\pi_1(S^{log}, \tilde{t})$ and a right action of $\pi_1(S^{log},\tilde{s})$.
Also, when $\bar{s} = \bar{t}$ then the set of comparison isomorphisms is equal to the image of the outer monodromy representation appearing in Proposition~\ref{s2:prop:ses_stable}.

Let $S'^{log}\to S^{log}$ be a morphism of fs log schemes and consider a cartesian diagram
\begin{equation}\label{s2:loc:base_change}
\begin{tikzcd}
\overline{X}{}'{}^{log}\arrow[r]\arrow[d] & X^{log}\arrow[d] \\
S'^{log}\arrow[r] & S^{log}.
\end{tikzcd}
\end{equation}
Choose two geometric points $\bar{s}'$ and $\bar{t}'$ of $S'^{log}$ lying over $\bar{s}$ and $\bar{t}$, respectively.
By considering two diagrams~\eqref{s2:diag:log_path_diagram} and a collection of arrows between them induced by the diagram~\eqref{s2:loc:base_change} we deduce that there is a natural injective map
\begin{equation}\label{s2:eq:relation_to_pullback}
C(\overline{X}{}'{}^{log}/S'^{log}, \bar{s}',\bar{t}')\hookrightarrow C(\overline{X}{}^{log}/S^{log}, \bar{s},\bar{t})
\end{equation}
In particular, together with Lemma~\ref{s2:lem:strictly_henselian} this implies that when $s$ specializes to $t$ then any specialization morphism
\[
\pi^{adm,\Sigma}_1(X_{\bar{s}}) \xrightarrow{\sim} \pi^{adm,\Sigma}_1(X_{\bar{t}})
\]
is also a comparison isomorphism.
The map~\eqref{s2:eq:relation_to_pullback} is compatible with both actions of the fundamental groups $\pi_1(S'^{log})$ and $\pi_1(S^{log})$ via the homomorphism 
\[
\pi_1(S'^{log}) \to \pi_1(S^{log}).
\]
In particular, when the above homomorphism of log fundamental groups is surjective then the map~\eqref{s2:eq:relation_to_pullback} is a bijection.

Assume now that $S^{log}$ is log regular and let $U\subset S$ be a nonempty open subscheme on which the log structure is trivial, thus $U$ is a regular scheme.
From the easier part of the log purity theorem we see that the map
\[
\pi_1(U)\twoheadrightarrow \pi_1(S^{log})
\]
is surjective.
The restriction of the stable log curve~\eqref{s2:eq:stable_log_curve} to $U$ determines a family of hyperbolic curves $X_U\to U$ and our discussion implies that for two geometric points $\bar{s}$, $\bar{t}$ of $U$ we have a bijection
\[
C(X_U/U, \bar{s},\bar{t})\xrightarrow{\sim} C(\overline{X}{}^{log}/S^{log}, \bar{s},\bar{t}).
\]
In other words, extending a family of hyperbolic curves over a regular scheme to a family of stable curves over a log regular log scheme does not change the set of comparison isomorphisms.

We now come back to the smooth case and we consider a family $\mathcal{F}$ of hyperbolic curves over $S$ as in Definition~\ref{s2:def:family_of_curves}.
Let $\bar{s}$ be a geometric point of $S$ and $U_{\bar{s}}\subset \pi^{\Sigma}_1(X_{\bar{s}})$ be an open subgroup.
We say that $U_{\bar{s}}$ is \emph{monodromy trivial} if there exists an open subgroup $U\subset \pi'_1(X)$ which surjects onto $\pi_1(S)$ such that 
\begin{equation}\label{s2:loc:U_s}
U_{\bar{s}} = U\cap \pi^{\Sigma}_1(X_{\bar{s}}).
\end{equation}
This implies that the action of $\pi_1(S)$ on the set of conjugacy classes of open subgroups of $\pi^{\Sigma}_1(X_{\bar{s}})$ fixes the conjugacy class of $U_{\bar{s}}$.
Assume now that $U_{\bar{s}}$ is monodromy trivial and let $\bar{t}$ be another geometric point of $S$.
Then for any comparison isomorphism
\[
\alpha \in C(\mathcal{F},\bar{s},\bar{t})
\]
the subgroup
\[
U_{\bar{t}} = \alpha(U_{\bar{s}})\subset \pi^{\Sigma}_1(X_{\bar{t}})
\]
is also monodromy trivial, moreover its conjugacy class is independent of the chosen comparison isomorphism $\alpha$.
Let $U\subset \pi'_1(X)$ be an open subgroup surjecting onto $\pi_1(S)$ satisfying~\eqref{s2:loc:U_s}, write $X'\to X$ for a finite \'etale cover corresponding to $U$.
Then the geometric fibres of the composition
\[
X'\to X\to S
\]
are connected and for every geometric point $\bar{t}$ of $S$ the finite \'etale cover $X'_{\bar{t}}\to X_{\bar{t}}$ corresponds to the subgroup $U_{\bar{t}}\subset \pi^{\Sigma}_1(X_{\bar{t}})$.
Finally, note that every open subgroup of $\pi^{\Sigma}_1(X_{\bar{s}})$ becomes monodromy trivial after replacing $S$ by a sufficiently large finite \'etale cover.

Next, we are going to describe how fibre comparison isomorphisms behave with respect to the base change along finite \'etale morphisms.  
Let $S'\to S$ be a connected finite \'etale cover and $\mathcal{F}'$ be the pullback of the family $\mathcal{F}$ to $S'$.
Let $\bar{s}'$ and $\bar{t}'$ be two geometric points of $S'$ lying over $\bar{s}$ and $\bar{t}$, respectively, then we have an injection
\begin{equation}\label{s2:loc:F'_to_F}
C(\mathcal{F}', \bar{s}', \bar{t}') \hookrightarrow C(\mathcal{F}, \bar{s}, \bar{t}).
\end{equation}
Fibres of $S'\to S$ over $\bar{s}$ and $\bar{t}$ have natural actions by fundamental groups $\pi_1(S,\bar{s})$ and $\pi_1(S,\bar{t})$ which are compatible with their action on the set of comparison isomorphisms, i.e., for every $\pi_s\in \pi_1(S,\bar{s})$ and $\pi_t\in \pi_1(S,\bar{t})$ we have an equality of sets 
\[
C(\mathcal{F}', \pi_s\bar{s}',\pi_t\bar{t}')  =  \pi_tC(\mathcal{F}', \bar{s}',\bar{t}')\pi_s,
\]
here we regard the injection~\eqref{s2:loc:F'_to_F} as an inclusion.
In particular, as these actions are transitive we obtain
\[
C(\mathcal{F},\bar{s},\bar{t}) = \bigcup_{\pi\in\pi_1(S,\bar{s})} C(\mathcal{F}', \pi\bar{s}',\bar{t}') = \bigcup_{\pi\in\pi_1(S,\bar{t})} C(\mathcal{F}', \bar{s}',\pi\bar{t}').
\]
Note that the two cases we have considered, namely a monodromy trivial subgroup and a pullback family, together describe comparison isomorphisms between fibres for a family constructed in diagram~\eqref{s2:diag:construction_of_a_family}.

In the final part of this section we introduce the notion of a cuspidal isomorphism and show that comparison isomorphisms for a family of hyperbolic curves are cuspidal.
Let $k$ be an algebraically closed field, for a hyperbolic curve $X$ over $k$ we write $g(X)$ for the genus of $X$ and $r(X)$ for the number of cusps.
Each cusp $c$ of $X$ determines a conjugacy class of cuspidal subgroups 
\[
I_c\subset \pi_1^{\Sigma}(X),
\]
with $I_c$ isomorphic noncanonically to $\Zhat^{\Sigma}$.
Every cuspidal subgroup $I_c$ is commensurably terminal in $\pi_1^{\Sigma}(X)$ (see \cite[Lem 1.3.7] {mochizuki_2004}), in particular each inertia subgroup $I_c$ is uniquely determined by any of its open subgroups.
Recall that as an abstract profinite group $\pi^{\Sigma}_1(X)$ is isomorphic to the pro-$\Sigma$ completion of the standard surface group $S_{g,r}$ given by the following generators and a unique relation 
\[
S_{g,r} = \langle a_1,\ldots, a_g, b_1, \ldots, b_g, c_1,\ldots , c_r | [a_1,b_1]\ldots [a_g,b_g]c_1\ldots c_r =1 \rangle,
\]
here $g=g(X)$ and $r=r(X)$.
In particular, the (topological) abelianization of $\pi^{\Sigma}_1(X)$ is a free abelian pro-$\Sigma$ group of rank $2g(X) + r(X) - 1$.
Moreover, note that when $r(X) > 0$ then $\pi^{\Sigma}_1(X)$ is a free pro-$\Sigma$ group hence in general the isomorphism class of $\pi^{\Sigma}_1(X)$ does not even determine the genus of $X$.

Let $X$ and $X_1$ be two hyperbolic curves over $k$ and consider an isomorphism
\[
\alpha \colon \pi^{\Sigma}_1(X)\xrightarrow{\sim} \pi^{\Sigma}_1(X_1)
\]
of profinite groups.

\begin{definition}\label{s2:def:cuspidal_isom}
We say that the isomorphism $\alpha$ is \textit{cuspidal} if $\alpha$ induces a bijection between the sets of conjugacy classes of cuspidal subgroups. 
\end{definition}

Note that compositions of cuspidal isomorphisms are cuspidal and trivially any inner automorphism is cuspidal, in particular we may talk about cuspidality of outer isomorphisms between pro-$\Sigma$ fundamental groups of hyperbolic curves.   
Observe that if $\alpha$ is cuspidal then we automatically have $r(X) = r(X_1)$ hence also $g(X) = g(X_1)$.
Cuspidal isomorphisms appear naturally as comparison isomorphisms between fibres in families of hyperbolic curves, as the next lemma shows.

\begin{lemma}
Let $\mathcal{F}$ be a family of hyperbolic curves over $S$ and let $\bar{s}$ and $\bar{t}$ be two geometric points of $S$.
Then any comparison isomorphism
\[
\pi^{\Sigma}_1(X_{\bar{s}})\xrightarrow{\sim} \pi^{\Sigma}_1(X_{\bar{t}})
\]
is cuspidal.
\end{lemma}

\begin{proof}
For specialization isomorphisms this is true by \cite[XIII, Lemme 2.11]{sga1}.
In general, by using a similar argument as at the end of the proof of Proposition~\ref{s2:prop:ses_stable} we reduce to the case $S$ is regular, in particular normal.
Then we may further reduce to the case when $S$ is the spectrum of a field and $\bar{s} = \bar{t}$ in which case the result is well known. 
\end{proof}

We close this section by giving an alternative definition of the notion of a cuspidal isomorphism, linking it to the theory which will be developed in the next section. Let 
\[
U\subset \pi^{\Sigma}_1(X)
\]
be an open subgroup corresponding to a finite \'etale cover $X'\to X$.
Using the isomorphism $\alpha$ we define an open subgroup 
\[
U_1 = \alpha (U) \subset \pi^{\Sigma}_1(X_1)
\]
corresponding to a finite \'etale cover $X'_1\to X_1$.
In the situation as above we say that the $\Sigma$-cover $X'\to X$ \textit{corresponds via} $\alpha$ to the $\Sigma$-cover $X'_1\to X_1$.

\begin{definition}
We say that $\alpha$ is \textit{numerically cuspidal} if $r(X') = r(X'_1)$ for every pair of $\Sigma$-covers corresponding via $\alpha$ (hence also $g(X') = g(X'_1)$ for every such a pair).
\end{definition}

Additionally, restricting $\alpha$ to $U$ we obtain another isomorphism denoted as
\[
\alpha_U \colon U = \pi^{\Sigma}_1(X') \xrightarrow{\sim} \pi^{\Sigma}_1(X'_1) = U_1.
\]
Then the basic properties of cuspidal isomorphisms are contained in the following lemma.

\begin{lemma}\label{s2:lem:cuspidal_isoms}
With the above notation:
\begin{enumerate}[(i)]
\item $\alpha$ is cuspidal if and only if $\alpha_U$ is cuspidal,
\item $\alpha$ is cuspidal if and only if it is numerically cuspidal,
\item Let $\Sigma'\subset \Sigma$ be a nonempty subset and 
\[
\alpha'\colon \pi^{\Sigma'}_1(X) \xrightarrow{\sim} \pi^{\Sigma'}_1(X_1)
\]
the isomorphism induced by $\alpha$ via the characteristic quotient $\pi^{\Sigma}_1(X)\twoheadrightarrow \pi^{\Sigma'}_1(X)$.
Then $\alpha$ is cuspidal if and only if $\alpha'$ is cuspidal. 
\end{enumerate}
\end{lemma}

\begin{proof}
The first statement follows from the fact that an inertia subgroup is uniquely determined by any of its open subgroups.
The second statement is proved in~\cite[Thm 1.6 (i)]{mochizuki_2007} while the third one is clear as the quotient $\pi^{\Sigma}_1(X)\twoheadrightarrow \pi^{\Sigma'}_1(X)$ induces a bijection between the corresponding sets of conjugacy classes of cuspidal subgroups.
\end{proof}

\section{Graphic isomorphisms}\label{s:graphic}
In this section we will discuss the notion of a graphic isomorphism between fundamental groups of hyperbolic curves. 
The main result of this section is Proposition~\ref{s3:prop:graphicity_numeric} which gives a numerical criterion for checking whether an isomorphism between fundamental group is graphic, similar to the one (Lemma~\ref{s2:lem:cuspidal_isoms}) obtained for cuspidal isomorphisms.

Let $p$ be a prime number and let $k$ be an algebraic closure of the fraction field of a strictly henselian discrete valuation ring with residue characteristic $p$.
Write $S$ for the spectrum of the ring of integers of $k$ and $\bar{s}$ for a geometric point determined by the closed point of $S$.
For a hyperbolic curve $X$ over $k$ we write $\Delta_X$ for the $p$-prime fundamental group of $X$.
By the stable reduction theorem $X$ extends uniquely to a stable model $\mathcal{X}$ over $S$, let $\Gamma(X)$ be the corresponding dual graph of the special fibre $\mathcal{X}_{\bar{s}}$ (see \cite[Appendix]{mochizuki_2004}).  
Write $V(X)$ for the set of vertices, $E(X)$ for the set of closed edges and $C(X)$ for the set of open edges of the graph $\Gamma(X)$.
We remind that $V(X)$ is in bijection with the set of irreducible components of the special fibre, $E(X)$ is in bijection with the set of nodes of the special fibre and the set of open edges $C(X)$ is in bijection with the set of cusps of $X$.
It will be convenient to introduce a labelling on the set $V(X)$ as follows: at a vertex $v\in V(X)$ we put a nonnegative integer equal to the genus of the normalization of $v$.  
In the following we will always consider the graph $\Gamma(X)$ as a labelled graph and call it
the \emph{reduction graph} of $X$; we remark that an isomorphism of reduction graphs is simply an isomorphism of graphs respecting labels.
We write
\[
e(X) = \# E(X), \; v(X) = \# V(X),
\]
for the cardinalities of the sets of edges and vertices of $\Gamma(X)$ and 
\[
t(X) = \rank_{\Z} H_1(\Gamma(X),\Z)
\]
for the rank of the first homology group of the geometric realization of $\Gamma(X)$.
We have the relation 
\[
t(X) = v(X) - e(X) + 1
\]
and we call $t(X)$ the \textit{toric rank} of $X$.
Denoting by $J$ the Jacobian of the smooth compactification of $X$ we know that $J$ extends uniquely to a semiabelian scheme $\mathcal{J}$ over $S$; then $t(X)$ is equal to the toric rank of the semiabelian variety $\mathcal{J}_{\bar{s}}$.
 
Let $\Sigma\subset \mathbb{P}\setminus \{p\} $ be a nonempty subset and $X$ be a hyperbolic curve over $k$.
We briefly recall the relationship between the pro-$\Sigma$ fundamental group $\Delta^{\Sigma}_X$ and the combinatorics of the dual graphs of covers of $X$.
Let $Y\to X$ be a $\Sigma$-cover, it uniquely extends to a finite morphism $\mathcal{Y}\to \mathcal{X}$ of stable models (see \cite[Proposition 2.2]{tamagawa_2004_resolution}) and induces a morphism of graphs $\Gamma(Y)\to \Gamma(X)$.
When $Y\to X$ is also Galois then the group $\Delta^{\Sigma}_X$ act naturally on $\Gamma(Y)$ and this action is compatible with towers of Galois covers.
In particular, the group $\Delta^{\Sigma}_X$ acts on the projective limit
\[
\tilde{\Gamma}(X) = \varprojlim \Gamma(Y),
\]
where $Y$ runs over all Galois $\Sigma$-covers and this action is faithful.
Given a vertex $v\in \Gamma(X)$ we choose a pro-vertex $\tilde{v}\in \tilde{\Gamma}(X)$ mapping to $v$ and we define a closed subgroup $D_{\tilde{v}}\subset \Delta^{\Sigma}_X$ as the stabilizer of $\tilde{v}$.
Different choices of $\tilde{v}$ lead to conjugate subgroups $D_{\tilde{v}}$ thus we obtain a conjugacy class of closed subgroups 
\[
D_v\subset \Delta^{\Sigma}_X,
\]
called \emph{verticial} subgroups.
Applying similar consideration to a closed edge $e\in \Gamma(X)$ we obtain a conjugacy class of closed subgroups
\[
I_{e}\subset \Delta^{\Sigma}_X,
\]
called \emph{nodal} subgroups.
Similarly to cuspidal subgroups, verticial and nodal subgroups are commensurably terminal in $\Delta^{\Sigma}_X$ (see \cite[Prop 1.2]{mochizuki_2007}); in particular they are uniquely determined by any of their open subgroups.
Let $e\in E(X)$ be an edge, for every branch $b$ of $e$ abutting to a vertex $v\in V(X)$ there is an injective homomorphism
\[
\iota_b \colon I_e \hookrightarrow D_v.
\]
Similarly, for a cusp $c$ lying on an irreducible component $v$ we have an injective homomorphism
\[
\iota_c \colon I_c \hookrightarrow D_v.
\]
The above data of groups $D_v, I_e, I_c$ and homomorphisms $\iota_b, \iota_c$ determines a graph of groups $\mathcal{G}^{\Sigma}(X)$ (see~\cite{serre1980trees}) whose fundamental group is naturally isomorphic to $\Delta^{\Sigma}_X$ (see~\cite[Appendix]{mochizuki_2004}).

Before continuing, let us link the above discussion to the content of Section~\ref{s:families}.
Let $k_0\subset k$ be the fraction field of a strictly henselian discrete valuation ring dominated by $k$ and $S_0$ be the spectrum of the ring of integers of $k_0$.
Assume that $\mathcal{X}$ descends to a stable model $\mathcal{X}_0$ over $S_0$.
Thus for some choice of an fs log structure on $S_0$ we obtain a stable log curve over $S^{log}_0$ which determines a specialization outer isomorphism
\[
\Delta^{\Sigma}_X\xrightarrow{\sim}  \pi^{adm,\Sigma}_1(\mathcal{X}_{\bar{s}})
\]
defined uniquely up to the outer action
\begin{equation}
\pi_1(S^{log}_0)\to \Out(\pi^{adm,\Sigma}_1(\mathcal{X}_{\bar{s}})).
\end{equation}
The outer action of $\pi_1(S^{log}_0)$ is given by profinite Dehn twists, thus preserves the conjugacy class of each verticial, nodal and cuspidal subgroup.
In particular, consider the log structure on $S_0$ determined by the closed point and take $\Sigma = \mathbb{P}\setminus \{p\}$, then the log fundamental group $\pi_1(S^{log}_0)$ is isomorphic to the tame inertia group of $k_0$.
Thus, writing $G_{k_0}$ for the absolute Galois group of $k_0$, the outer action
\[
G_{k_0}\to \Out(\Delta_X)
\]
factorizes through the maximal tame quotient which acts on $\Delta_X$ via profinite Dehn twists.

Next, we are going to introduce some \'etale covers of $X$ having particularly simple combinatorial description.
Let $Y\to X$ be a Galois finite \'etale cover of degree $d$ which is prime to $p$ and denote by
\[
\phi \colon V(Y)\to V(X)
\]
the induced map on the sets of vertices.
For a vertex $v\in V(X)$ we say that $v$ is \emph{totally split} in $Y$ if $\phi^{-1}(v)$ has cardinality $d$ and that $v$ is \emph{inert} in $Y$ if $\phi^{-1}(v)$ has cardinality one.
It is easy to check that $v$ is totally split if and only if a decomposition subgroup $D_v$ lies in the kernel of the map
\[
\Delta_X \twoheadrightarrow \Gal(Y/X).
\]
Similarly, $v$ is inert if and only if $D_v$ surjects onto $\Gal(Y/X)$.

Denote by $\Delta^{et}_X$ the maximal $p$-prime quotient of the \'etale fundamental group of the compactification of $\mathcal{X}_{\bar{s}}$, thus the specialization homomorphism determines a quotient map
\begin{equation}\label{s3:loc:eq:Delta_et}
\Delta_X\twoheadrightarrow \Delta^{et}_X.
\end{equation}
The kernel of this quotient is the normal subgroup of $\Delta_X$ generated by cuspidal and nodal subgroups, \'etale covers classified by this quotient will be called \textit{unramifed over cusps and nodes}.
The group $\Delta^{et}_X$ admits a further quotient
\begin{equation}\label{s3:loc:eq:Delta_cmb}
\Delta^{et}_X\twoheadrightarrow \Delta^{cmb}_X
\end{equation}
whose kernel is generated by the image of verticial subgroups; \'etale covers classified by this quotient will be called \emph{combinatorial}.
The group $\Delta^{cmb}_X$ is a free $p$-prime profinite group of rank $t(X)$.

For a hyperbolic curve $X$ we say that its reduction is \emph{sturdy} if the normalization of every irreducible component of $\mathcal{X}_{\bar{s}}$ has genus at least two; also we say that its reduction is \emph{untangled} if every closed edge in $E(X)$ abuts to two distinct vertices. 
Every connected finite \'etale cover of a hyperbolic curve with sturdy (resp. untangled) reduction also has sturdy (resp. untangled) reduction; moreover for every hyperbolic curve there exists a characteristic connected finite \'etale cover which has sturdy and untangled reduction.
Later in this section we will need to construct certain $\Sigma$-covers of $X$ with specified number of vertices and edges, this is achieved by the next lemma.

\begin{lemma}\label{s3:lem:cover_split_on_S}
Let $X$ be a hyperbolic curve over $k$ with sturdy reduction and $N\ge 1$ be a natural number.
Fix a proper subset $S\subsetneq V(X)$ of the set of vertices of the reduction graph of $X$.
Then there exists a connected Galois finite \'etale cover $Y\to X$ with Galois group $\Z/N\Z$, unramified over cusps and nodes, such that for $v\in V(X)$ we have that $v\in S$ (resp. $v\not\in S $) if and only if $v$ is totally split (resp. inert) in $Y$.

In particular, we have $e(Y) = N e(X)$ and $v(Y)  =  v(X) + (N-1) \#S$.
\end{lemma}

\begin{proof}
Such a cover can be constructed directly from the combinatorial description of $\Delta_X$; alternatively we could refer to~\cite[Lem 1.4]{hoshi_mochizuki_2011} which proves a more general statement 
\end{proof}

Assume now that $X$ is a hyperbolic curve with sturdy reduction, following~\cite[\S 1]{mochizuki_1996} we are going to recall a construction of a canonical bijection between the set of vertices $V(X)$ and a certain subset of homomorphisms from the fundamental group $\Delta^{\Sigma}_X$.
The essential point of this construction is that from the profinite group $\Delta^{\Sigma}_X$ equipped with some combinatorial data one can construct a set equipped with a canonical bijective map to the set $V(X)$ of vertices of $X$.

Choose a prime number $\ell\in \Sigma$ and consider a connected Galois finite \'etale cover $Y\to X$ of degree $\ell$, unramified over cusps.
We always have the inequality
\[
v(Y)\le \ell v(X)
\]
and the equality holds if and only if $Y\to X$ is a combinatorial cover.
Indeed, the equality 
\[
v(Y) = \ell v(X)
\]
implies that $Y\to X$ is totally split over every vertex of $X$ thus for each vertex $v\in V(X)$ a decomposition group $D_v$ map trivially to $\Gal(Y/X)$.

Next we assume additionally that $X$ has untangled reduction.
Suppose that the cover $Y\to X$ is not combinatorial, then we have 
\begin{equation}\label{s3:loc:eq:maximal}
v(Y)\le \ell (v(X)-1) + 1.
\end{equation}
We say that the cover $Y\to X$ is \emph{maximal} if it is unramified over cusps and the equality holds in~\eqref{s3:loc:eq:maximal}.
Then $Y\to X$ is maximal if and only if it is unramifed over cusps and nodes and moreover there exists a unique vertex in $V(X)$ which is inert in $Y$ with the rest of vertices being totally split.
Indeed, the equality 
\[
v(Y) = \ell (v(X)-1) + 1
\]
implies that there exist a subset $S\subset V(X)$ of cardinality $v(X) - 1$ such that all vertices in $S$ are totally split in $Y$ and the remaining vertex is inert.
Since the reduction of $X$ is untangled, for every edge $e\in V(X)$ there exists a vertex $v\in S$ and a branch homomorphism $\iota_b\colon I_e \hookrightarrow D_v$.
Therefore as $v$ is totally split in $Y$ we see that the cover $Y\to X$ is unramifed over cusps and nodes.

Consider the set 
\[
M_{\ell}(X)\subset \Hom(\Delta^{et}_X, \Z/\ell\Z)
\]
of all surjective homomorphism whose kernels correspond to maximal covers.
There is a natural map
\begin{equation}\label{s3:eq:M_to_V}
M\to V(X),
\end{equation}
which sends a maximal cover to the unique vertex which is inert.
For a vertex $v\in V(X)$ we apply Lemma~\ref{s3:lem:cover_split_on_S} to $S = V(X)\setminus \{v\}$ to see that the map~\eqref{s3:eq:M_to_V} is surjective.
On $M_{\ell}(X)$ we define an equivalence relation $\sim$ as follows: for $\phi,\phi'\in M_{\ell}(X)$ we write $\phi\sim \phi'$ if and only if for all $\lambda,\lambda'\in (\Z/\ell\Z)^{\times}$, whenever 
\[
\psi = \lambda \phi + \lambda' \phi' \ne 0
\]
and the cover determined by the kernel $\psi$ is not combinatorial then it is also maximal.
One checks~\cite[Prop 1.3]{mochizuki_1996} that $\sim$ is indeed an equivalence relation, moreover writing
\[
C_{\ell}(X) = M_{\ell}(X)/\sim
\]
the surjective map~\eqref{s3:eq:M_to_V} descends to a bijection 
\begin{equation}\label{s3:eq:C_to_V}
C_{\ell}(X) \xrightarrow{\sim} V(X).
\end{equation}

Next, we discuss the functoriality of the above construction.
Given a $\Sigma$-cover $Y\to X$ corresponding to an open subgroup $\Delta_Y\subset \Delta_X$ we may construct a map $C_{\ell}(Y)\to C_{\ell}(X)$ as follows.
Let $\bar{\phi}\in C_{\ell}(X)$, choose a representative $\phi \in M_{\ell}(X)$ of $\bar{\phi}$ such that the restriction of $\phi$ to $\Delta_Y$ 
\[
\phi_Y \in \Hom(\Delta_Y,\Z/\ell\Z)
\]
is surjective.
Then there exist pairwise nonequivalent homomorphisms $\phi_i\in M_{\ell}(Y)$ for $1\le i\le s$ such that the kernel of the homomorphism
\[
\psi = \phi_Y +\phi_1 + \ldots + \phi_s
\]
corresponds to a combinatorial cover.
We define the map $C_{\ell}(Y)\to C_{\ell}(X)$ by sending the class of $\phi_i$ to the class of $\phi$, for every $1\le i\le s$.
One can check that this rule is well defined and determines a map $C_{\ell}(Y)\to C_{\ell}(X)$ such that the diagram
\[
\begin{tikzcd}
C_{\ell}(Y)\arrow[r, "\sim"]\arrow[d] & V(Y) \arrow[d]\\
C_{\ell}(X)\arrow[r, "\sim"]  & V(X)
\end{tikzcd}
\]
commutes.

After these preparations we come to the definition of a graphic isomorphism.
Let $X$ and $X_1$ be two hyperbolic curves over $k$ and $\Sigma\subset \mathbb{P}\setminus \{p\}$ be a nonempty set of prime numbers.
Write $\Delta$ for the $p$-prime fundamental group of $X$ and $\mathcal{G}^{\Sigma}$ for the graph of groups $\mathcal{G}^{\Sigma}(X)$; we also denote by $\Delta_1$ and $\mathcal{G}^{\Sigma}_1$ the corresponding data for $X_1$.
Let
\begin{equation}\label{s3:eq:alpha}
\alpha\colon \Delta^{\Sigma} \xrightarrow{\sim} \Delta^{\Sigma}_1
\end{equation}
be an isomorphism of profinite groups.

\begin{definition}\label{s3:def:veg_isom}
We say that $\alpha$ is
\begin{enumerate}[(i)]
\item \emph{verticial} if it induces a bijection between the sets of conjugacy classes of verticial subgroups,
\item \emph{nodal} if it induces a bijection between the set of conjugacy classes of nodal subgroups,
\item \emph{graphic} if it arises from an isomorphism of graphs of groups $\mathcal{G}^{\Sigma}\xrightarrow{\sim} \mathcal{G}^{\Sigma}_1$.
\end{enumerate}
\end{definition}

Clearly a graphic isomorphism is also cuspidal, nodal and verticial, the next proposition says that the converse statement holds as well.
Also note that any inner automorphism is graphic thus we may also use Definition~\ref{s3:def:veg_isom} for outer isomorphisms.

\begin{proposition}\label{s3:prop:graphicity=cuspidal_and_verticial}
With the above definitions we have:
\begin{enumerate}[(i)]
\item $\alpha$ is graphic if and only if it cuspidal, verticial and nodal,
\end{enumerate}
Moreover, if $\alpha$ is verticial then it is also nodal hence $(i)$ may be strengthened to
\begin{enumerate}[(i)]
\setcounter{enumi}{1}
\item $\alpha$ is graphic if and only if it cuspidal and verticial.
\end{enumerate}
\end{proposition}

\begin{proof}
The property $(i)$ is shown in~\cite[Prop 1.5 (ii)]{mochizuki_2007} whereas the fact that verticial isomorphisms are automatically nodal is shown in~\cite[Prop 1.13]{hoshi_mochizuki_2011}.
\end{proof}

For an open subgroup $U\subset \Delta^{\Sigma}$ determining a finite \'etale cover $Y\to X$ we write $U_1 = \alpha(U)$ and denote by $Y_1\to X_1$ a finite \'etale cover of $X_1$ corresponding to $Y$ via $\alpha$.
We extend our notation from curves to open subgroups and write
\[
v(U) = v(Y),\; e(U) = e(Y), \; t(U) = t(Y), \; V(U) = V(Y),
\]
similarly for $U_1$ and $Y_1$.
Denote by 
\begin{equation}\label{s3:eq:alpha_U}
\alpha_U \colon U\xrightarrow{\sim} U_1
\end{equation}
the restriction of $\alpha$ to $U$.
Let $\Sigma'\subset \Sigma$ is a nonempty subset, after passing to the maximal pro-$\Sigma'$ quotient~\eqref{s3:eq:alpha} induces the isomorphism  
\[
\alpha'\colon \Delta^{\Sigma'} \xrightarrow{\sim} \Delta^{\Sigma'}_1,
\]
whereas~\eqref{s3:eq:alpha_U} induces the isomorphism
\[
\alpha'_U \colon U^{\Sigma'}\xrightarrow{\sim} U^{\Sigma'}_1.
\]
The next lemma relates the graphicity of $\alpha'$ and $\alpha_U'$. 
\begin{lemma}\label{s3:lem:graphicity_restriction}
With the above notation we have:
\begin{enumerate}[(i)]
\item If $\alpha_U'$ is graphic then $\alpha'$ is graphic as well.
\item Suppose that $U\subset \Delta^{\Sigma}$ is a $\Sigma'$-subgroup. Then $\alpha_U'$ is graphic if and only if $\alpha'$ is graphic.
\end{enumerate}
In particular, $\alpha$ is graphic if and only if $\alpha_U$ is graphic.
\end{lemma}
\begin{proof}
The property $(i)$ follows from Proposition~\ref{s3:prop:graphicity=cuspidal_and_verticial} and the fact that verticial (and nodal) subgroups are uniquely determined by any of their open subgroups.
To prove $(ii)$ note that under the assumption that $U\subset \Delta^{\Sigma}$ is a $\Sigma'$-subgroup the natural map
\[
U^{\Sigma'} \to \Delta^{\Sigma'}
\]
is an open inclusion.
Thus the statement follows again from Proposition~\ref{s3:prop:graphicity=cuspidal_and_verticial}.
\end{proof}
Similarly as in the case of cuspidal isomorphisms the property of being graphic can be given an equivalent formulation as an equality of certain combinatorial data for all open subgroups.
\begin{definition}
We say that $\alpha$ is
\begin{enumerate}[(i)]
\item \emph{numerically verticial} if for every open subgroup $U\subset \Delta^{\Sigma}$ we have $v(U) = v(U_1)$,
\item \textit{numerically toric} if for every open subgroup $U\subset \Delta^{\Sigma}$ we have $t(U) = t(U_1)$.
\end{enumerate}
\end{definition}
Then we have the following proposition linking the graphicity of an isomorphism to the numerical properties of corresponding open subgroups.
\begin{proposition}\label{s3:prop:graphicity_numeric}
Assume that $\alpha$ is cuspidal.
Then the following conditions are equivalent:
\begin{enumerate}[(i)]
\item $\alpha$ is graphic,
\item $\alpha$ is numerically verticial,
\item $\alpha$ is numerically toric.
\end{enumerate}
\end{proposition}
Before giving the proof we need to prove two lemmas. 
The first one relates the properties of being numerically vertical and numerically toric.   

\begin{lemma}\label{s3:lem:toric_to_ev}
Assume that $X$ and $X_1$ have sturdy reduction and $t(X) = t(X_1)$.
Suppose that there exists a prime number $\ell$ such that for all normal subgroups $U\subset \Delta$ of $\ell$-power index we have the equality
\[
t(U) = t(U_1).
\]
Then we also have $v(X) =  v(X_1)$.

\end{lemma}
\begin{proof}
It is enough to show the equality $e(X) = e(X_1)$.
Let $n\ge 1$ be a natural number and $Y\to X$ be a connected finite \'etale cover of degree $\ell^n$ obtained from Lemma~\ref{s3:lem:cover_split_on_S} for the set $S$ being empty.
Thus we have
\begin{equation}\label{s3:loc:eq:special_cover}
e(Y) = \ell^n e(X),\; v(Y) = v(X).
\end{equation}
Consider the connected finite \'etale cover $Y_1\to X_1$ corresponding via $\alpha$ to $Y$.
From our assumptions we have $t(Y) = t(Y_1)$ hence
\begin{equation}\label{s3:loc:eq:toric_ranks}
e(Y) - v(Y) = e(Y_1) - v(Y_1).
\end{equation}
On the other hand the map of graphs $\Gamma(Y_1) \to \Gamma(X_1)$ gives the inequalities
\begin{equation}\label{s3:loc:eq:bound}
e(Y_1)\le \ell^n e(X_1),\; v(Y_1)\ge v(X_1).
\end{equation}
Equations~\eqref{s3:loc:eq:special_cover} and~\eqref{s3:loc:eq:toric_ranks} together with inequalities~\eqref{s3:loc:eq:bound} imply that
\[
\ell^n e(X) - v(X)\le \ell^n e(X_1) - v(X_1),
\]
therefore we obtain
\[
\ell^n(e(X) - e(X_1)) \le v(X) - v(X_1).
\]
By increasing $n$ we conclude that $e(X)\le e(X_1)$ thus $e(X) = e(X_1)$ by symmetry.
\end{proof}

Recall from our discussion following Lemma~\ref{s3:lem:cover_split_on_S} that, 
assuming hyperbolic curves $X$ and $X_1$ have sturdy and untangled reduction, for every prime number $\ell\in \Sigma$ we have defined the subset of maximal covers
\[
M_{\ell}(X)\subset \Hom(\Delta^{\Sigma}, \Z/\ell\Z),
\]
equipped with an equivalence relation $\sim$, and a natural bijection 
\[
C_{\ell}(X) \xrightarrow{\sim} V(X).
\]
Write $\sim_1$ for the corresponding equivalence relation on $M_{\ell}(X_1)$. 
Since the above data was defined using the geometry of reductions of curves $X$ and $X_1$ it is not necessarily preserved by the isomorphism $\alpha$.
However we can force the compatibility under some combinatorial assumptions, as the next lemma shows.

\begin{lemma}\label{s3:lem:functorial_M}
Assume that $X$ and $X_1$ have sturdy and untangled reduction and that $v(X) = v(X_1)$.
Let $\ell \in \Sigma$ be a prime number such that for every normal open subgroup $U\subset \Delta^{\Sigma}$ of index $\ell$ we have the equality $v(U) = v(U_1)$. 
Then the isomorphism $\alpha$ naturally induces a bijection between $M_{\ell}(X)$ and $M_{\ell}(X_1)$ which is compatible with equivalence relations $\sim$ and $\sim_1$. 
\end{lemma}
\begin{proof}
Recall that a connected Galois finite \'etale cover $Y\to X$ of degree $\ell$ is maximal if and only if it is unramified over cusps and satisfies the equality
\[
v(Y) = \ell (v(X) - 1) + 1.
\]
Since $\alpha$ is cuspidal and preserves the number of vertices for all such covers we see that the isomorphism $\alpha$ induces a bijection 
\[
M_{\ell}(X)\xrightarrow{\sim} M_{\ell}(X_1).
\]
Finally, the equivalence relations $\sim$ and $\sim_1$ are derived from the notions of combinatorial and maximal covers which in turn are characterized purely in terms of the numbers of vertices $v(Y)$ and $v(X)$.
Thus it follows from our assumptions that they are preserved by $\alpha$.  
\end{proof}

\begin{proof}[Proof of Proposition~\ref{s3:prop:graphicity_numeric}]
Clearly we have the implications $(i)\Rightarrow (ii)$ and $(i)\Rightarrow (iii)$, thus we need to prove that each of the conditions $(ii)$ and $(iii)$ implies the graphicity of $\alpha$.
Thanks to Lemma~\ref{s3:lem:graphicity_restriction} we may pass to an open subgroup and assume that $X$ and $X_1$ have sturdy and untangled reduction.
Assume $(iii)$, then applying Lemma~\ref{s3:lem:toric_to_ev} to all open subgroups of $\Delta^{\Sigma}$ we deduce that $\alpha$ is numerically verticial.
Thus it is enough to show that $(ii)$ implies $(i)$.

Assume $(ii)$ and fix a prime number $\ell\in \Sigma$.
Given an open subgroup $U\subset \Delta^{\Sigma}$ we write $C_{\ell}(U) = C_{\ell}(Y)$ where $Y\to X$ is a $\Sigma$-cover determined to $U$, similarly for covers of $X_1$ corresponding via $\alpha$.
Applying Lemma~\ref{s3:lem:functorial_M} to all open subgroups of $\Delta$ we deduce that $\alpha$ induces an isomorphism in the middle of the following diagram
\[
\varprojlim V(U)\xleftarrow{\sim} \varprojlim C_{\ell}(U) \xrightarrow{\sim} \varprojlim C_{\ell}(U_1)\xrightarrow{\sim} \varprojlim V(U_1),
\]
here the inverse limits are over all open subgroups $U\subset \Delta$.
The above isomorphisms are compatible with the natural actions of $\Delta$ and $\Delta_1$ via $\alpha$.
Therefore by looking at the stabilizers of pro-vertices on both sides we deduce that $\alpha$ is verticial.
Hence $\alpha$ is graphic by Proposition~\ref{s3:prop:graphicity=cuspidal_and_verticial}.   
\end{proof}

From Proposition~\ref{s3:prop:graphicity_numeric} we immediately obtain the following corollary.
\begin{corollary}\label{s3:cor:toric_separation_for_two}
Suppose that $\alpha$ is cuspidal and the reduction graphs $\Gamma(X)$ and $\Gamma(X_1)$ are not isomorphic. Then there exists an open subgroup $U\subset \Delta$ such that $t(U)\ne t(U_1)$.
\end{corollary}

\section{A criterion for graphicity}\label{s:criterion}

In this section we will state and prove Theorem~\ref{s4:thm:graphicity_criterion} which is a generalization of Proposition~\ref{s3:prop:graphicity_numeric} to an arbitrary finite number of hyperbolic curves.
Then using a compactness argument we derive Corollary~\ref{s4:cor:compactness} which is the crucial result used in Section~\ref{s:separating_classes}.
Thus we have to generalize the situation considered in Section~\ref{s:graphic} to allow for a finite number of hyperbolic curves.
We keep the same assumptions as in the previous section.

Let $n\ge 1$ be a natural number and let 
\begin{equation}\label{s4:eq:collection_of_hyp_curves}
X_0, X_1,\ldots ,X_n
\end{equation}
be a collection of hyperbolic curves over $k$.
To simplify the notation we will use the convention of omitting the lower subscript `$0$' from our notation for objects corresponding to the hyperbolic curve $X = X_0$.
For $0\le i\le n$ write $\Delta_i =\Delta_{X_i} $ for the $p$-prime fundamental group of $X_i$.
Suppose that for $1\le i \le n$ we are given cuspidal isomorphisms
\[
\alpha_i \colon \Delta \xrightarrow{\sim} \Delta_i,
\]
by convention we write $\alpha_0$ for the identity isomorphism on $\Delta = \Delta_0$.
For every pair of indices $0\le i,j\le n$ we write 
\[
\alpha_{ij} \colon \Delta_i \xrightarrow{\sim} \Delta_j
\]  
for the isomorphism obtained as the composition $\alpha_j\circ\alpha_i^{-1}$.
Given an open subgroup $U\subset \Delta$ we write $U_i\subset \Delta_i$ for the corresponding open subgroup defined as 
\[
U_i = \alpha_{i}(U),
\]
for every $0\le i\le n$.
Hence $U_i$ determines a connected finite \'etale cover $Y_i\to X_i$; as before we will say that covers $Y_i$ and $Y_j$ correspond via $\alpha_{ij}$, for all $0\le i,j\le n$.
We will always regard the group $U_i$ as an open subgroup of $\Delta_i$ and for convenience we write
\[
e(U_i) = e(Y_i), \; v(U_i) = v(Y_i),\; t(U_i) = t(Y_i),\; r(U_i) = r(Y_i).
\]
Finally, for a natural number $n\ge 1$ denote by $P_n$ the set of all unordered pairs $(i,j)$ of distinct integers with $0\le i\ne j\le n$ and let $P\subset P_n$ be a subset. 
After these preparations we may state our desired generalization of Proposition~\ref{s3:prop:graphicity_numeric} which is the main result of this section.

\begin{theorem}\label{s4:thm:graphicity_criterion}
Assume that for every open subgroup $U\subset \Delta$ there exists $(i,j)\in P$ such that 
\[
t(U_i) = t(U_j).
\]
Then there exists $(i,j)\in P$ such that the isomorphism 
\[
\alpha_{ij}\colon \Delta_i\xrightarrow{\sim} \Delta_{j}
\]
is graphic.
\end{theorem}

As a corollary of Theorem~\ref{s4:thm:graphicity_criterion} we obtain a generalization of Corollary~\ref{s3:cor:toric_separation_for_two} from the previous section.

\begin{corollary}\label{s4:cor:toric_separation_for_n}
Assume that for every $(i,j)\in P$ the reduction graphs of $X_i$ and $X_j$ are not isomorphic.
Then there exists an open subgroup $U\subset \Delta$ such for every $(i,j)\in P$ we have $t(U_i) \ne t(U_j)
$.
\end{corollary}

\begin{proof}
Indeed, since for every $(i,j)\in P$ the reduction graphs $X_i$ and $X_j$ are not isomorphic it follows that for every pair $(i,j)\in P$ the isomorphism $\alpha_{ij}$ is not graphic.
Thus we conclude using Theorem~\ref{s4:thm:graphicity_criterion}.
\end{proof}

To prove Theorem~\ref{s4:thm:graphicity_criterion} it will be enough to prove the following seemingly weaker statement.

\begin{proposition}\label{s4:prop:graphicity_criterion_ell}
With the same assumptions an in Theorem~\ref{s4:thm:graphicity_criterion}, there exists $(i,j)\in P$ such that the isomorphism 
\[
\alpha^{\ell}_{ij}\colon \Delta^{\ell}_i\xrightarrow{\sim} \Delta^{\ell}_{j}
\]
is graphic for all sufficiently large prime numbers $\ell$.
\end{proposition}

First we show how to derive Theorem~\ref{s4:thm:graphicity_criterion} from Proposition~\ref{s4:prop:graphicity_criterion_ell}.

\begin{proof}[Proof of Theorem~\ref{s4:thm:graphicity_criterion}]
For an open subgroup $U\subset \Delta$ we define a subset $S_U\subset P$ as the set of all pairs $(i,j)\in P$ such that, writing $\alpha = \alpha_{ij}$, for all sufficiently large prime numbers $\ell$ the isomorphism 
\[
\alpha^{\ell}_U\colon U^{\ell}_i\xrightarrow{\sim} U^{\ell}_j
\]
is graphic.
With this definition Proposition~\ref{s4:prop:graphicity_criterion_ell} is equivalent to the statement that the set $S_{\Delta}$ is nonempty.
Applying Proposition~\ref{s4:prop:graphicity_criterion_ell} to all open subgroups $U\subset \Delta$ we deduce that for every $U$ the set $S_U$ is nonempty.
Moreover, thanks to Lemma~\ref{s3:lem:graphicity_restriction} we see that for every pair of open subgroups 
\[
U\subset V\subset \Delta
\]
we have the inclusion $S_U\subset S_V$.
Thus the intersection
\[
S = \bigcap  S_U
\]
is nonempty, let $(i,j)\in S$ and write $\alpha = \alpha_{ij}$.
Then for every open subgroup $U\subset \Delta$ there exists a natural number $n_U$ such that for all prime numbers $\ell > n_U$ the isomorphism $\alpha^{\ell}_U$ is graphic, in particular this implies that
\[
v(U_i) = v(U_j).
\]
Hence $\alpha$ is numerically verticial and by Proposition~\ref{s3:prop:graphicity_numeric} we conclude that $\alpha$ is graphic. 
\end{proof}

The proof of Proposition~\ref{s4:prop:graphicity_criterion_ell} will use an induction argument on $n$.
Note that in the case $n=1$ the equivalence $(i)\Leftrightarrow (iii)$ in Proposition~\ref{s3:prop:graphicity_numeric} clearly implies Theorem~\ref{s4:thm:graphicity_criterion}. 
However, for the induction to work properly, we will first need to look more closely at the case of two hyperbolic curves.  
Thus, for the next lemma we assume that $n=1$ and write $\alpha = \alpha_1$.

\begin{lemma}\label{s4:lem:V_U_subgroups}
Assume that $X$ and $X_1$ have sturdy and untangled reduction.
Let $\ell\ne p$ be a prime number such that the induced isomorphism 
\[
\alpha^{\ell} \colon \Delta^{\ell} \xrightarrow{\sim} \Delta_1^{\ell}
\]
is not graphic.
Then there exists a normal open subgroup $U\subset \Delta$ of $\ell$-power index such that for every prime number $\ell'$ there is a normal open subgroup $V\subset U$ of index dividing $\ell'$ such that
$v(V) \ne v(V_1)$.
\end{lemma}

\begin{proof}
Since $\alpha^{\ell}$ preserves cusps and is not graphic it follows from Proposition~\ref{s3:prop:graphicity_numeric} that there exists a normal open subgroup $U\subset \Delta$ of $\ell$-power index and an intermediate subgroup $U\subset H\subset \Delta$ such that 
\[
v(H) \ne v(H_1).
\]
If $v(U) \ne v(U_1)$ then we may take $V = U$ so assume that $v(U) = v(U_1)$.

Suppose on the contrary that there exists a prime number $\ell'$ such that for every normal open subgroup $V\subset U$ of index $\ell'$ we have $v(V) = v(V_1)$.
Then it follows from Lemma~\ref{s3:lem:functorial_M} applied to $\ell'$ and the open subgroup $U$ that $\alpha$ induces a bijection
\[
V(U)\xrightarrow{\sim} V(U_1),
\] 
which is compatible with the action of $\Delta/U$ on both sides.
However, after taking $H/U$-invariants we obtain
$v(H) = v(H_1)$ which is a contradiction.
\end{proof}

We now come back to the general case $n\ge 1$.
To establish the induction step in the proof of Proposition~\ref{s4:prop:graphicity_criterion_ell} we will use the following key lemma.

\begin{lemma}\label{s4:lem:preserving_distinct_values}
There exists a natural number $N$ with the following property: for every prime number $\ell > N$, every normal open subgroup $U\subset \Delta$ of $\ell$-power index and every pair $(i,j)\in P_n$, whenever $t(X_i)\ne t(X_j)$ then
\[
t(U_i) \ne t(U_j)
\]
as well.
\end{lemma}

\begin{proof}
Observe that it is enough to prove the lemma in the case of only two curves, thus we assume that $n=1$ and $t(X) \ne t(X_1)$.
We will show that the statement holds with the following constant 
\[
N = e(X) + e(X_1) + v(X) + v(X_1).
\]
Let $\ell > N$ be a prime number and $m\ge 1$ be a natural number.
Consider a normal open subgroup $U\subset \Delta$ of index $\ell^m$.
Note that for every vertex $v\in V(X)$ the preimage of $v$ under the morphism of graphs $\Gamma(U)\to \Gamma(X)$ is a set of cardinality $\ell^j$ for some $0\le j \le m$.
For $i =0,1$ and $0\le j\le m$ denote by $v_{ij}$ (resp. $e_{ij}$) the cardinality of the set of vertices (resp. edges) of $\Gamma(X_i)$ whose preimage through the map $\Gamma(U_i)\to \Gamma(X_i)$ has size $\ell^{j}$.
Then for $i=0,1$ we have 
\[
e(X_i) = \sum^{m}_{j=0} e_{ij},\quad v(X_i) = \sum^{m}_{j=0} v_{ij},
\]
as well as
\[
e(U_i) = \sum^{m}_{j=0}e_{ij}\ell^{j},\quad v(U_i) = \sum^{m}_{j=0}v_{ij}\ell^{j}.
\]  
Suppose that $t(U) = t(U_1)$ thus also $e(U) - v(U) = e(U_1) - v(U_1)$ hence
\[
\sum^{m}_{j=0}(e_{0j} - v_{0j})\ell^{j} = \sum^{m}_{j=0}(e_{1j} - v_{1j})\ell^{j}.
\] 
Putting $a_j = (e_{0j} - v_{0j}) - (e_{1j} - v_{1j})$ for $0\le j\le m$ we obtain
\begin{equation}\label{s4:loc:ell_power_sum}
a_0 + a_1\ell+ \ldots + a_m\ell^m = 0.
\end{equation}
However $|a_j|\le N \le \ell - 1$ and this inequality together with equation~\eqref{s4:loc:ell_power_sum} implies that $a_j=0$ for all $0\le j\le m$.
Hence for every $0\le j\le m$ we have 
\[
e_{0j} - v_{0j} = e_{1j} - v_{1j}.
\]
Taking the sum over all $0\le j\le m$ we deduce that $e(X) - v(X) = e(X_1) - v(X_1)$ hence $t(X) = t(X_1)$, contrary to our assumption.
Hence we must have $t(U)\ne t(U_1)$ which shows that the natural number $N$ has the desired property.
\end{proof}

Finally, we can put together all the pieces to prove our main result of this section.

\begin{proof}[Proof of Proposition~\ref{s4:prop:graphicity_criterion_ell}]
We will prove the contrapositive; namely we assume that for all $(i,j)\in P$ the isomorphism $\alpha^{\ell}_{ij}$ is not graphic for infinitely many primes $\ell$ and we construct an open subgroup $U\subset \Delta$ such that for every $(i,j)\in P$ we have $t(U_i)\ne t(U_j)$.
In the proof we will repeatedly replace $\Delta$ by an open subgroup; note that by Lemma~\ref{s3:lem:graphicity_restriction} our assumption is stable with respect to this operation.
In particular we may replace $\Delta$ by a characteristic open subgroup and assume that all curves $X_i$ for $0\le i\le n$ have sturdy and untangled reduction.
We proceed by induction on $n$, the base case $n = 1$ follows from Proposition~\ref{s3:prop:graphicity_numeric}.
Thus let $n\ge 2$ and assume that the statement holds for all sets of hyperbolic curves of cardinality at most $n$

Let $P'_n\subset P_n$ be the subset of all unordered pairs $(i,j)$ with $1\le i\ne j\le n$ and put 
\[
P' = P\cap P'_n.
\]
Write $Q$ for the set of all integers $0\le j\le n$ such that the pair $(0,j)$ belongs to $P$.
Then we have a disjoint sum decomposition
\[
P = P'\sqcup Q',
\]
where $Q'\subset P_n$ is the set of pairs $(0,j)$ with $j\in Q$.
We apply the induction hypothesis to the collection of hyperbolic curves 
\[
X_1,\ldots ,X_n
\]
and the set $P'$.
Thus, after replacing $\Delta$ by an open subgroup we may assume that for every pair $(i,j)\in P'$ we have 
\[
t(X_i)\ne t(X_j).
\]
Write $R\subset Q$ for the set of all $j\in Q$ such that $t(X)\ne t(X_j)$ and let $R'\subset Q'$ be the set of all pairs $(0,j)\in Q'$ with $j\in R$.
From these definitions we see that the following property holds:
\begin{quote}
$(*)$ for every pair $(i,j)$ belonging to $P'\sqcup R'$ we have $t(X_i)\ne t(X_j)$.
\end{quote}
If $R = Q$ then $R' = Q'$ hence we are already done thus we may assume that $R\ne  Q$.
Changing labels we may also assume that $1\in Q\setminus R$, write $\alpha = \alpha_{1}$.
We are going to show that after passing to a suitable open subgroup of $\Delta$ the property $(*)$ will still hold and additionally we will get $t(X)\ne t(X_1)$; in other words the cardinality of the set $R$ will increase.
This will finish the proof, as in a finite number of steps we will reach the situation when $R = Q$.

We apply Lemma~\ref{s4:lem:preserving_distinct_values} to the collection of hyperbolic curves
\[
X_0, X_1,\ldots , X_n
\]
to find the corresponding constant $N$.
By our assumption there exists a prime number $\ell > N$ such that the isomorphism $\alpha^{\ell}$ is not graphic.
Hence it follows from Lemma~\ref{s4:lem:V_U_subgroups} that we can find a normal open subgroup $U\subset \Delta$ of $\ell$-power index satisfying the following property:
\begin{quote}
$(**)$ for every prime number $\ell'$ there exists a normal open subgroup $V\subset U$ of index dividing $\ell'$ such that $v(V)\ne v(V_1)$.
\end{quote}
Replacing $\Delta$ by the open subgroup $U$ we may assume that $\Delta$ itself satisfies the property $(**)$.
Moreover by the conclusion of Lemma~\ref{s4:lem:preserving_distinct_values} we see that after this replacement the property $(*)$ still holds.

We again apply Lemma~\ref{s4:lem:preserving_distinct_values} as above and choose a new constant $N$.
Choose a prime number $\ell'> N$ and let $V\subset \Delta$ be a normal open subgroup of index dividing $\ell'$ obtained from the property $(**)$.
Replacing $\Delta$ by $V$ as previously we may therefore assume that
\[
v(X) \ne v(X_1),
\]
moreover as before the property $(*)$ still holds.
If $t(X) \ne t(X_1)$ then we have already enlarged the set $R$, thus we may assume that $t(X) = t(X_1)$.

We apply Lemma~\ref{s4:lem:preserving_distinct_values} for the third time to find yet another constant $N$.
Choose a prime number $\ell > N$. 
Then it follows from Lemma~\ref{s3:lem:toric_to_ev} that there exists a normal open subgroup $U\subset \Delta$ of $\ell$-power index such that $t(U) \ne t(U_1)$.
Replacing $\Delta$ by the open subgroup $U$ we see as before that the property $(*)$ still holds, moreover by the choice of $U$ we have enlarged the set $R$.
As we have already stated this is enough to conclude. 
\end{proof} 

\begin{remark}
In fact, our arguments in this section used only covers of hyperbolic curves obtained as a composition of cyclic covers of prime power degree. 
Therefore, it follows that Theorem~\ref{s4:thm:graphicity_criterion} and Corollary~\ref{s4:cor:toric_separation_for_n} are valid also in the situation when the groups $\Delta_i$ are replaced by their maximal pro-solvable quotients.
\end{remark}

In the final part of this section we are going to apply Theorem~\ref{s4:thm:graphicity_criterion} together with a compactness argument to obtain Corollary~\ref{s4:cor:compactness} which is an extension of Corollary~\ref{s4:cor:toric_separation_for_n}; this result will be essential for the proofs of Theorems~\ref{s1:thm:separation_in_char_zero} and~\ref{s1:thm:separation_by_toric_rank}.

Consider again the collection~\eqref{s4:eq:collection_of_hyp_curves} of hyperbolic curves over $k$, for some natural number $n\ge 1$.
Define $P\subset P_n$ to be the subset consisting of all pairs $(i,j)\in P$ such that the reduction graphs of hyperbolic curves $X_i$ and $X_j$ are not isomorphic. 
Write $C$ for the set of all $n$-tuples 
\[
\alpha =(\alpha_1,\ldots \alpha_n)
\]
of cuspidal isomorphisms
\[
\alpha_i\colon \Delta\xrightarrow{\sim} \Delta_i.
\]
Given an $n$-tuple $\alpha$ and an open subgroup $U\subset \Delta$ we write $U^{\alpha}_i$ for $\alpha_i(U)$; additionally we denote by $\alpha_0$ the identity map on $\Delta$.

Since $\Delta$ is topologically finitely generated, there exists a descending filtration of characteristic open subgroups
\[
\Delta\supset U^{(1)}\supset U^{(2)}\supset \ldots,
\]
such that the intersection $\cap_{j\ge 1} U^{(j)}$ is trivial.
Note that for every $1\le i\le n$ we also get an analogous filtration
\[
\Delta_i \supset U^{(1)}_i \supset U^{(2)}_i \supset \ldots
\]
such that any isomorphism $\Delta\xrightarrow{\sim}\Delta_i$ is strictly compatible with these filtrations.
For an integer $j\ge 1$ we write $A_j$ for the set of all $n$-tuples
\[
\bar{\alpha} = (\bar{\alpha}_1,\ldots \bar{\alpha}_n)
\]
of isomorphisms
\[
\bar{\alpha}_i \colon \Delta/U^{(j)} \xrightarrow{\sim} \Delta_i/U^{(j)}_i,
\]
note that all sets $A_j$ are finite.
Given an $n$-tuple $\bar{\alpha} \in A_j$ and an open subgroup $\Delta\supset U\supset U^{(j)}$ we write $U^{\bar{\alpha}}_i$ for the open subgroup of $\Delta_i$ determined by $\bar{\alpha}_i(U)$; additionally we put $U^{\bar{\alpha}}_0 = U$. 
For an integer $j\ge 1$ we define a subset $C_j\subset A_j$ consisting of all $n$-tuples $\bar{\alpha}\in A_j$ such that for every open subgroup 
\[
\Delta\supset U\supset U^{(j)}
\]
we have $r(U) = r(U^{\bar{\alpha}}_i)$.
Observe that there are natural restriction maps $C\to C_j$ for $j\ge 1$ and $C_{j'}\to C_{j}$ for $j'\ge j$.
These restriction maps are clearly compatible with each other thus we have an induced map 
\[
C\to \varprojlim C_j
\]
which is a bijection by Lemma~\ref{s2:lem:cuspidal_isoms},~$(ii)$. 
Next, for every $j\ge 1$ we define a subset $S_j\subset C_j$ consisting of all $n$-tuples $\bar{\alpha}\in C_j$ such that for every open subgroup 
\[
\Delta\supset U\supset U^{(j)}
\]
there exist indices $(i,i')\in P$ such that $t(U^{\bar{\alpha}}_i) = t(U^{\bar{\alpha}}_{i'})$.
Clearly for integers $j'\ge j$ the natural map $C_{j'}\to C_j$ restricts to a map $S_{j'}\to S_j$.
 
We claim that the inverse limit
\begin{equation}\label{s4:loc:inverse_limit}
\varprojlim S_j
\end{equation}
is empty. Indeed, suppose this is not the case and let
\[
\alpha \in \varprojlim S_j \subset \varprojlim C_j = C
\]
be an $n$-tuple of cuspidal isomorphisms.
For every open subgroup $U\subset \Delta$ there exists an integer $j\ge 1$ such that
\[
\Delta\supset U \supset U^{(j)}
\]
Write $\bar{\alpha}$ for the image of $\alpha$ in $S_j$.
By the definition of the set $S_j$ we see that there exist indices $(i,i')\in P$ such that
\[
t(U^{\alpha}_i) = t(U^{\bar{\alpha}}_i) = t(U^{\bar{\alpha}}_{i'}) = t(U^{\alpha}_{i'})
\]
As the open subgroup $U$ was arbitrary it follows from Theorem~\ref{s4:thm:graphicity_criterion} that there exist indices $(i,i')\in P$ such that the isomorphism
\[
\alpha_{ii'} = \alpha_{i'} \circ \alpha^{-1}_i \colon \Delta_i\xrightarrow{\sim} \Delta_{i'}
\]
is graphic.
But this is a contradiction since by the definition of $P$ the reduction graphs of $X_{i}$ and $X_{i'}$ are not isomorphic.
Thus the inverse limit~\eqref{s4:loc:inverse_limit} is empty, since all sets $S_j$ are finite this implies that there exists an integer $N\ge 1$ such that the sets $S_j$ are empty for all $j\ge N$.
We summarize the above discussion in the next corollary.

\begin{corollary}\label{s4:cor:compactness}
For a finite collection~\eqref{s4:eq:collection_of_hyp_curves} of hyperbolic curves over $k$ there exists a natural number $N\ge 1$ with the following property: for any $n$-tuple $\alpha \in C$ of cuspidal isomorphisms there exists an open subgroup $U\subset \Delta$ of index at most $N$ such that whenever the reduction graphs of $X_i$ and $X_j$ are not isomorphic then the toric ranks
\[
t(U^{\alpha}_i)\ne t(U^{\alpha}_j)
\]
are distinct.
\end{corollary}

\section{Separating classes by the toric rank}\label{s:separating_classes}

In this section we are going to apply the theory developed in previous sections to construct families of curves separating classes of points, our main result is Proposition~\ref{s5:prop:separating_compact_set}.
The situation we consider will be slightly more general than in Section~\ref{s:introduction}, as a finite set of points will be replaced by a finite partition of a set of points.  
We keep the same assumptions and notation as in the previous two sections.

First we introduce some notation and terminology.
For a set $S$ and a collection of subsets $S_i\subset S$ for $i\in I$ we say that 
\[
\mathcal{S} = \{S_i\}_{i\in I}
\]
is a \emph{partition} of $S$ if the subsets $S_i$ are pairwise disjoint, nonempty and $S = \cup_i S_i$.
Clearly a partition of $S$ is the same as an equivalence relation on $S$ and we will refer to the subsets $S_i$ as \emph{classes} of the partition $\mathcal{S}$.
A subset $S_0\subset S$ is called a \emph{set of representatives} of $\mathcal{S}$ if $S_0$ contains exactly one element from each class of $\mathcal{S}$.
When $I$ is finite we say that $\mathcal{S}$ is a \emph{finite} partition, when every class of $\mathcal{S}$ has exactly one element we say that $\mathcal{S}$ is a \emph{trivial} partition.
Given another partition 
\[
\mathcal{T} = \{T_j \}_{j\in J}
\]
of $S$ we say that $\mathcal{T}$ is \emph{finer} than $\mathcal{S}$ if for every $j\in J$ there exists $i\in I$ such that $T_j\subset S_i$.
For a set $S'$ and a map $\varphi\colon S'\to S$ the nonempty subsets $\varphi^{-1}(S_i)\subset S'$ define a partition of $S'$ called the \emph{pullback} of the partition $\mathcal{S}$.
For a hyperbolic curve $X$ over $k$, a subset $S\subset X(k)$ and a connected finite \'etale cover $X'\to X$ we write $S|_{X'}$ for the preimage of $S$ under the map $X'(k)\to X(k)$, when $S$ is equipped with a partition $\mathcal{S}$ then the partition of $S|_{X'}$ obtained as the pullback of $\mathcal{S}$ is denoted by $\mathcal{S}|_{X'}$.

Let $X$ be a hyperbolic curve over $k$.
Recall that in Definition~\ref{s1:def:separation_by_toric_rank} we introduced the notion of a family of proper curves separating points of a finite set by the toric rank.
Here we generalize it to the case of a finite set of families of hyperbolic curves and a finite partition of a not necessarily finite set.

\begin{definition}
Let $S\subset X(k)$ be a subset with a finite partition $\mathcal{S}$ and $\{Y_j\}_{j\in J}$ be a finite set of families of hyperbolic curves over $X$.
We say that families $\{Y_j\}_{j\in J}$ \emph{separate classes} of $\mathcal{S}$ by the \emph{toric rank} if for every set of representatives $S_0\subset S$ of $\mathcal{S}$
there exists an index $j\in J$ such that the compactification of the family $Y_j$ separates points in $S_0$ by the toric rank.
\end{definition}

Clearly this recovers Definition~\ref{s1:def:separation_by_toric_rank} in the case of a finite set and the trivial partition.
For convenience we state the following lemma whose proof is immediate from the definition.

\begin{lemma}\label{s5:lem:finer_partition}
Let $\{Y_j\}_{j\in J}$ be a finite set of families of hyperbolic curves over $X$ separating classes of $\mathcal{S}$ by the toric rank.
Let $\mathcal{T}$ be another partition of $S$ and suppose that $\mathcal{S}$ is finer than $\mathcal{T}$. 
Then families $\{Y_j\}_{j\in J}$ separate classes of $\mathcal{T}$ by the toric rank.
\end{lemma}

From now on we assume that $X$ is a proper hyperbolic curve. Write 
\[
D\subset X\times_k X
\]
for the diagonal and
\[
Y = X\times_k X\setminus D
\]
for the second configuration space of $X$.
We consider $Y$ as a scheme over $X$ via the first projection thus $Y\to X$ is a family of hyperbolic curves over $X$ with a fibre over a point $s\in X(k)$ equal to
\[
Y_s = X\setminus\{s\},
\]
we also write $\Delta_s = \Delta_{Y_s}$ for the $p$-prime fundamental group of the fibre over $s$.
We are going to define an equivalence relation on the set $X(k)$ of $k$-rational points of $X$.

\begin{definition} 
Let $s,t\in X(k)$ be two $k$-rational points of $X$.
We say that $s$ and $t$ are \emph{weakly graphically} equivalent if the reduction graphs of hyperbolic curves $Y_s$ and $Y_t$ are isomorphic.
Classes of the resulting partition of $X(k)$ are called \emph{weakly graphic}.
\end{definition}

Note that the partition of $X(k)$ into weakly graphic classes is a finite partition.
Using our results on graphic isomorphisms from Section~\ref{s:criterion} and the notion of a comparison isomorphism from Section~\ref{s:families} we prove our first result concerning separating classes by the toric rank.

\begin{proposition}\label{s5:prop:separating_finite_set}
Let $S\subset X(k)$ be a finite subset and $\mathcal{S}$ be a partition of $S$ into weakly graphic classes.
Then there exists a connected finite \'etale cover $X'\to X$ and a finite set of families of hyperbolic curves $\{Y_j\}_{j\in J}$ over $X'$ which separate classes of $\mathcal{S}|_{X'}$ by the toric rank.

Moreover, for every index $j\in J$ there exists a commutative diagram
\[
\begin{tikzcd}[cramped]
Y_j\arrow[r]\arrow[rd] & Y'\arrow[r]\arrow[d] & Y\arrow[d] \\
 & X'\arrow[r] & X,
\end{tikzcd}
\]
with horizontal morphisms being finite \'etale and the square being cartesian.
\end{proposition}

\begin{proof}
When $\#S = 1$ then the statement is trivial thus assume that $\#S = n + 1$ with $n\ge 1$.
Denote the elements of $S$ as
\[
S = \{s_0, s_1,\ldots , s_n\}.
\]
Denote by $Y_i = Y_{s_i}$ the fibre of $Y\to X$ over $s_i$ and let $\Delta_i$ be the $p$-prime fundamental group of $Y_i$; we use the same convention of omitting the label `$0$' as in Section~\ref{s:criterion}. 
Write $P\subset P_n$ for the set of all unordered pairs $(i,j)$ of integers $0\le i\ne j  \le n$ such that $(i,j)\in P$ if and only if $s_i$ and $s_j$ are not weakly graphically equivalent.
We apply Corollary~\ref{s4:cor:compactness} to the collection of hyperbolic curves
\[
Y_0, Y_1, \ldots Y_n.
\]
Thus, there exists a natural number $N\ge 1$ such that for any $n$-tuple 
\[
\alpha = (\alpha_1,\ldots \alpha_n)
\]
of cuspidal isomorphisms
\[
\alpha_i \colon \Delta \xrightarrow{\sim} \Delta_i
\]
there exists an open subgroup $U\subset \Delta$ of index at most $N$ such that whenever $(i,j)\in P$ then the toric ranks
\[
t(U^{\alpha}_i) \neq t(U^{\alpha}_j)
\] 
are distinct; here we write $U^{\alpha}_i$ for $\alpha_i(U)$.
Let $X'\to X$ be a connected finite \'etale cover which trivializes the monodromy of all open subgroups of $\Delta$ of index at most $N$.
We are going to show that the classes of the partition $\mathcal{S}|_{X'}$ can be separated by the toric rank, which will finish the proof.

Consider a cartesian diagram
\[
\begin{tikzcd}
Y'\arrow[r] \arrow[d]& Y \arrow[d] \\
X'\arrow[r] & X
\end{tikzcd}
\]
and choose a lift $S'\subset X'(k)$ of $S$ to $X'$, also for every $0\le i\le n$ we denote by $s'_i\in S'$ the point mapping to $s_i\in S$.
For every $1\le i\le n$ choose a comparison isomorphism
\[
\alpha_i\in C(Y'/X', s'_0, s'_i),
\]
then from our discussion in Section~\ref{s:families} we see that 
\[
\alpha = (\alpha_1,\ldots, \alpha_n)
\]
may be considered as an $n$-tuple of cuspidal isomorphisms
\[
\alpha_i \colon \Delta \xrightarrow{\sim} \Delta_{i}.
\]
Therefore it follows that there exists an open subgroup $U\subset \Delta$ of index at most $N$ such that for every $(i,j)\in P$ the toric ranks
\begin{equation}\label{s5:loc:toric_ranks_diff}
t(U^{\alpha}_i)\neq t(U^{\alpha}_j)
\end{equation}
are distinct.
Since $X'$ was chosen to trivialize the monodromy of all open subgroups of $\Delta$ of index at most $N$ we see from our discussion in Section~\ref{s:families} that a finite \'etale cover determined by the open subgroup $U\subset \Delta $ extends to a finite \'etale cover
\[
Y''\to Y'
\]
of degree $\le N$ such that the composition $Y''\to X'$ is a family of hyperbolic curves.
Moreover for every $1\le i\le n$ the fibre of $Y''\to X'$ over $s'_i$ is a connected finite \'etale cover of $Y_i$ determined by the open subgroup $U^{\alpha}_i\subset \Delta_i$.
Therefore the inequality~\eqref{s5:loc:toric_ranks_diff} for every $(i,j)\in P$ implies that the family $Y''\to X'$ separates points in $S'$ by the toric rank.
This finishes the proof as there are only finitely many open subgroups of $\Delta$ of bounded index.
\end{proof}

To improve Proposition~\ref{s5:prop:separating_finite_set} further it will be necessary to introduce another equivalence relation on the set $X(k)$, we consider again the family $Y\to X$ of hyperbolic curves.
Recall that we have defined in Section~\ref{s:families} a subset
\[
C(Y/X,s,t) \subset \Isom^{\Out}(\Delta_s,\Delta_t)
\]
of comparison isomorphisms between fibres over $s$ and $t$.

\begin{definition} 
Let $s,t\in X(k)$ be two $k$-rational points. We say that points $s$ and $t$ are \emph{graphically} equivalent if there exists a comparison isomorphism
\[
\alpha\in C(Y/X,s,t)
\]
which is graphic.
Classes of the resulting partition are called \emph{graphic}.
\end{definition}

Note that points which are graphically equivalent are also weakly graphically equivalent thus the partition of $X(k)$ into graphic classes is finer than the partition into weakly graphic classes.
The next proposition gives a criterion for checking whether two $k$-rational points of $X$ are graphically equivalent.

\begin{proposition}\label{s5:prop:graphical_equivalence_criterion}
Let $s,t\in X(k)$ be two $k$-rational points.
Suppose that there exists a semistable model $\mathcal{X}$ of $X$ such that the reductions of $s$ and $t$ in the special fibre of $\mathcal{X}$ lie in the smooth locus of the same irreducible component.
Then points $s$ and $t$ are graphically equivalent.
\end{proposition}

\begin{proof}
Let $K\subset k$ be the fraction field of a strictly henselian discrete valuation ring dominated by $k$ and $X_K$ be a hyperbolic curve over $K$ which is isomorphic to $X$ after the base change to $k$.
By extending $K$ we may assume that $X_K$ has stable reduction over $K$ and that both points $s$ and $t$ arise from $K$-rational points of $X_K$.
Write $Y_K$ for the second configuration space of $X_K$ and let $Y_K\to X_K$ be the morphism determined by the first projection.
Then we have a cartesian diagram
\[
\begin{tikzcd}
Y\arrow[r]\arrow[d] & Y_K\arrow[d] \\
X\arrow[r] & X_K,
\end{tikzcd}
\] 
which determines an injection between the sets of comparison isomorphisms
\begin{equation}\label{s5:loc:comparison_1}
C(Y/X,s,t)\hookrightarrow C(Y_K/X_K,s,t).
\end{equation}
Our goal is to show that the left hand side of~\eqref{s5:loc:comparison_1} contains a graphic isomorphism.
We claim that it is enough to find a graphic isomorphism inside the right hand side of~\eqref{s5:loc:comparison_1}.
Indeed, suppose that 
\[
\alpha \in C(Y_K/X_K,s,t)
\] 
is graphic, since the group $G_K$ acts on the group $\Delta_s$ through graphic outer automorphisms we see that the $G_K$-orbit of $\alpha$ consists entirely of graphic automorphisms.
Thus the claim follows from the fact that the intersection of the $G_K$-orbit of $\alpha$ with the left hand side of~\eqref{s5:loc:comparison_1} is nonempty.

Choose a semistable model $\mathcal{X}_{\Oc_K}$ of $X_K$ over $\Oc_K$ such that the reductions of both points $s$ and $t$ lie on the smooth locus of the same irreducible component $v$ of the special fibre of $\mathcal{X}_{\Oc_K}$; this is possible by our assumption.
Write $\mathcal{U}$ for an open subscheme of $\mathcal{X}_{\Oc_K}$ obtained by removing all irreducible components from the special fibre apart from $v$.
Then $\mathcal{U}$ is a regular scheme and there exists a unique family of stable curves $\mathcal{Y}\to \mathcal{U}$ with a cartesian diagram
\[
\begin{tikzcd}
Y_K\arrow[r]\arrow[d] & \mathcal{Y} \arrow[d] \\
X_K \arrow[r] & \mathcal{U}.
\end{tikzcd}
\]
Furthermore the special fibre of $\mathcal{U}$ is a regular divisor hence determines a log regular log structure $\mathcal{U}^{log}$ on $\mathcal{U}$ making the compactification $\overline{\mathcal{Y}}\to\mathcal{U}$ into a stable log curve.
Thus it follows from our discussion in Section~\ref{s:families} that we have an equality
\begin{equation}\label{s5:loc:comparison_2}
C(Y_K/X_K,s,t) = C(\overline{\mathcal{Y}}^{log}/\mathcal{U}^{log},s,t)
\end{equation}
hence we are reduced to finding a graphic isomorphism inside the right hand side of~\eqref{s5:loc:comparison_2}.

Write $\bar{u}$ and $\bar{v}$ for geometric points lying over the reductions of $s$ and $t$ on the special fibre of $\mathcal{U}$ and $\bar{\eta}$ for a geometric point over the generic point of the special fibre of $\mathcal{U}$.
Then we have a diagram of specialization isomorphisms
\[
\begin{tikzcd}[column sep = tiny]
 & \pi^{adm, \Sigma}_1(\mathcal{Y}_{\bar{u}}) & & \pi^{adm,\Sigma}_1(\mathcal{Y}_{\bar{v}}) \\
\Delta_s\arrow[ru, "\sim" sloped] & & \pi^{adm,\Sigma}_1(\mathcal{Y}_{\bar{\eta}})\arrow[ru, "\sim" sloped, "\alpha_2"']\arrow[lu, "\sim" sloped, "\alpha_1"] & & \Delta_{t}\arrow[lu, "\sim" sloped].
\end{tikzcd}
\]
It follows from the construction that the isomorphisms $\alpha_1$ and $\alpha_2$ arise from isomorphisms of graphs of groups of corresponding fibres, therefore the isomorphism $\Delta_s\cong \Delta_t$ obtained from the above diagram is graphic.
\end{proof}

Choose a discrete valuation field $K\subset k$ dominated by $k$ and a hyperbolic curve $X_K$ over $K$ such that $X$ is the base change of $X_K$ along $K\subset k$.
For a subset
\[
S\subset X(k) = X_K(k)
\]
we say that $S$ is \emph{relatively compact} if there exists a finite field extension $L/K$ such that $S\subset X_K(L)$; clearly this notion is independent of the choice of the field $K$.
From Proposition~\ref{s5:prop:graphical_equivalence_criterion} we immediately obtain the next corollary. 

\begin{corollary}\label{s5:cor:finiteness_of_graphicity_classes}
For a relatively compact subset $S\subset X(k)$ the partition of $S$ into graphic classes is finite.
\end{corollary}

\begin{proof}
Choose a discrete valuation field $K\subset k$ dominated by $k$ and a hyperbolic curve $X_K$ over $K$ such that $S\subset X_K(K)$.
By enlarging $K$ we may assume that $X_K$ has stable model over $K$.
Let $\mathcal{X}^{reg}$ be the minimal regular model of $X_K$ over $\Oc_K$ and $\mathcal{X}$ be its base change to $\Oc$; it is a semistable model of $X$.
By construction the reduction of every point from $S$ to the special fibre of $\mathcal{X}$ lies in the smooth locus of the special fibre.
As the set of irreducible components of the special fibre of $\mathcal{X}$ is finite, the result follows from Proposition~\ref{s5:prop:graphical_equivalence_criterion}.
\end{proof}

Using Corollary~\ref{s5:cor:finiteness_of_graphicity_classes} we can prove a stronger version of Proposition~\ref{s5:prop:separating_finite_set} by weakening the assumption on the set $S$ from being finite to being relatively compact.

\begin{proposition}\label{s5:prop:separating_compact_set}
Let $S\subset X(k)$ be relatively compact subset and let $\mathcal{S}$ be a finite partition of $S$ into weakly graphic classes.
Then there exists a connected finite \'etale cover $X'\to X$ and a finite set of families of hyperbolic curves $\{Y_j\}_{j\in J}$ over $X'$ which separate classes in $\mathcal{S}|_{X'}$ by the toric rank.
\end{proposition}

\begin{proof}
Consider the partition of $S$ into graphic classes and let $T\subset S$ be a set of representatives of this partition; by Corollary~\ref{s5:cor:finiteness_of_graphicity_classes} we see that $T$ is a finite set.
Write $\mathcal{T}$ for the partition of $T$ into weakly graphic classes.
We apply Proposition~\ref{s5:prop:separating_finite_set} to the set $T$ and the partition $\mathcal{T}$ thus there exists a connected finite \'etale cover $X'\to X$ and a finite set of families of hyperbolic curves $\{Y_j\}_{j\in J}$ over $X'$ separating classes of $\mathcal{T}|_{X'}$ by the toric rank.
Additionally, for every $j\in J$ there exists a commutative diagram
\[
\begin{tikzcd}[cramped]
Y_j\arrow[r]\arrow[rd] & Y'\arrow[r]\arrow[d] & Y\arrow[d] \\
 & X'\arrow[r] & X,
\end{tikzcd}
\]
with horizontal morphisms being finite \'etale and the square being cartesian.
We will show that the set of families $\{Y_j\}_{j\in J}$ separates classes of $\mathcal{S}|_{X'}$, this will finish the proof.

Let $S_0\subset S$ be a set of representatives of $\mathcal{S}$ and let $S'_0\subset X'(k)$ be a lift of $S$, it is a set of representatives of $\mathcal{S}|_{X'}$.
We need to find an index $j\in J$ such that the family $Y_j$ separates points in $S'_0$.
Let $s'\in S'_0$ be a point with the image $s\in S_0$, there exists a unique point $t\in T$ such that $s$ and $t$ are graphically equivalent.
Recall from Section~\ref{s:families} that we have

\[
C(Y/X,s,t) =  \bigcup_{t'} C(Y'/X', s', t').
\]
where $t'$ runs through all points in the fibre of $X'\to X$ over $t$.
Therefore we may choose a point 
\[
\varphi(s')\in T|_{X'}
\]
lying over $t$ such that there exists a graphic comparison isomorphism
\[
\alpha_{s'}\in C(Y'/X', s', \varphi(s'))
\]
between fibres over $s'$ and $\varphi(s')$ for the family $Y'\to X'$. Write 
\[
T' = \{\varphi(s')\colon s'\in S'_0\}
\]
for a subset of $T|_{X'}$ containing the chosen elements $\varphi(s')$ for all $s'\in S'_0$, note that the constructed map 
\[
\varphi\colon S'_0\to T'
\]
is a bijection.
Observe that elements of $T'$ lie in pairwise different classes of the partition $\mathcal{T}|_{X'}$ thus we can find an index $j\in J$ such that the family $Y_j$ separates points in $T'$ by the toric rank.
On the other hand, as the family $Y_j\to X'$ arises from a finite \'etale cover of $Y'$ and $\alpha_{s'}$ is graphic it follows from our discussion in Section~\ref{s:families} that fibres of $Y_j\to X'$ over $s'$ and $\varphi(s')$ have isomorphic reduction graphs.
In particular for every $s'\in S'_0$ the toric ranks of fibres of $Y_j\to X'$ over $s'$ and $\varphi(s')$ are equal.
This implies that the family $Y_j$ separates points in $S'_0$ by the toric rank. 
\end{proof}

\section{Separating points}\label{s:separating_points}

In this section we will apply our results from Section~\ref{s:separating_classes} together with a resolution of nonsingularities to prove Theorems~\ref{s1:thm:separation_in_char_zero} and~\ref{s1:thm:separation_by_toric_rank}.
We keep the same assumptions as in the previous section, additionally in this section we assume that $k$ has characteristic zero and the residue field $\kappa$ of $k$ is an algebraic closure of the finite field $\mathbb{F}_p$.

Let $X$ be a proper hyperbolic curve over $k$.
Recall from Section~\ref{s:graphic} that any connected finite \'etale cover $X'\to X$ extends to a finite morphism $\mathcal{X}'\to \mathcal{X}$ between the stable models and induces a morphism of the dual graphs $\Gamma(X')\to \Gamma(X)$.
When the degree of $X'\to X$ is divisible by $p$ then the morphism $\mathcal{X}'\to \mathcal{X}$ in general is not finite, equivalently it is not quasi-finite so it must contract an irreducible component of the special fibre to a point.
Given a connected finite \'etale cover $X' \to X$ and an irreducible component $v'\in \Gamma(X')$ we say that $v'$ is \emph{vertical} if its image in $\mathcal{X}$ is a closed point, otherwise we say that $v'$ is \emph{horizontal}.
Write 
\[
V(X'/X)\subset V(X')
\]
for the set of vertical components and $R(X'/X)\subset \mathcal{X}(\kappa)$ for the image of $V(X'/X)$ in $\mathcal{X}$.
We will need the following result, called a resolution of nonsingularities.

\begin{theorem}[Tamagawa, {\cite[Theorem 0.2 (v)]{tamagawa_2004_resolution}}]
\label{s6:thm:tamagawa_rns}
For every $r\in \mathcal{X}(\kappa)$ there exists a connected finite \'etale cover $X'\to X$ such that $r\in R(X'/X)$.
\end{theorem}

Let $R\subset \mathcal{X}(\kappa)$ be a finite subset, in the next lemma we will use Theorem~\ref{s6:thm:tamagawa_rns} to construct a particular finite \'etale cover of $X$.

\begin{lemma}\label{s6:lem:special_cover_X'}
There exists a connected finite \'etale cover $\varphi \colon X'\to X$ with sturdy reduction having the following properties:
\begin{enumerate}[(i)]
\item we have $R\subset R(X'/X)$,
\item the fibre $\varphi^{-1}(x)$ over every point $x\in R(X'/X)$ is a union of verticial components,
\item any two vertical components contained in different fibres of $\varphi$ have different genera,
\item the sets of genera of vertical and horizontal components are disjoint.
\end{enumerate}
\end{lemma}

\begin{proof}
First we apply Theorem~\ref{s6:thm:tamagawa_rns} a finite number of times and pass to the Galois closure to construct a connected finite \'etale cover $X'\to X$ satisfying the property $(i)$. 
Observe that in the Galois case the property $(ii)$ holds automatically, indeed, by considering the Stein factorization of the morphism $\mathcal{X}'\to \mathcal{X}$ we see that the fibre over $x$ is a disjoint sum of connected curves with the Galois group acting transitively on the set of connected components.
Next, note that for every $p$-prime cover $X''\to X'$ the set of vertical components $V(X''/X)$ of the composition $X''\to X$ is the preimage of the set $V(X'/X)$ under the map $V(X'')\to V(X')$, moreover we have the equality
\[
R(X''/X) = R(X'/X).
\]
In particular, the composite cover $X''\to X$ still has properties $(i)$ and $(ii)$.
Hence to finish the proof we only need to modify the cover $X'\to X$ by a $p$-prime cover to obtain properties $(iii)$ and $(iv)$.

We may assume that the reduction of $X'$ is sturdy.
Write 
\[
R(X'/X) = \{x_1,\ldots , x_n\}
\]
for some $n\ge 1$.
Denote by $H\subset V(X')$ the set of horizontal components and by 
\[
V_i\subset V(X'/X)
\]
the set of all vertical components lying over $x_i$, for $1\le i \le n$.
Denote by $h$ the maximum of genera of horizontal components and by $a_i$ (respectively, $b_i$) the minimum (respectively, the maximum) of genera of components belonging to $V_i$, for $1\le i\le n$.
We are going to show that after replacing $X'$ by a suitable $p$-prime cover we will obtain the sequence of inequalities
\begin{equation}\label{s6:loc:eq:genera}
h < a_1\le b_1 <a_2 \le b_2 < \ldots <a_n\le b_n.
\end{equation}
This will clearly imply both properties $(iii)$ and $(iv)$.

To achieve this we will repeatedly apply Lemma~\ref{s3:lem:cover_split_on_S}, choosing a suitable subset $S\subset V(X)$ and replacing $X'$ by a cyclic cover of sufficiently large degree.
First, we apply Lemma~\ref{s3:lem:cover_split_on_S} with $S = H$ and $N$ sufficiently large to obtain the inequality $h < a_1$.
Next, suppose we already have the inequality
\[
h < a_1\le b_1 <a_2 \le b_2 < \ldots <a_m\le b_m.
\]
for some $m < n$.
We apply Lemma~\ref{s3:lem:cover_split_on_S} again, with 
\[
S = H\cup V_1\cup\ldots \cup V_m
\]
and choose $N$ sufficiently large to obtain $b_m < a_{m+1}$.
Thus after finitely many steps we reach the inequality~\eqref{s6:loc:eq:genera}.
\end{proof}

Using the reduction map on the stable model of $X$ we may define another partition of the set of $k$-rational points of $X$.

\begin{definition}
A partition of $X(k)$ obtained by taking the pullback of the trivial partition of $\mathcal{X}(\kappa)$ along the reduction map
\[
X(k)\to \mathcal{X}(\kappa),
\]
will be called a partition into \emph{stable} classes.
\end{definition}

Let $S\subset X(k)$ be a relatively compact subset and write $\mathcal{S}$ for the partition of $S$ into stable classes.
Since the image of $S$ under the reduction map is finite we see that $\mathcal{S}$ is a finite partition. 
The next lemma relates the partition into stable classes to the partition into weakly graphic classes considered in Section~\ref{s:separating_classes}.

\begin{lemma}\label{s6:lem:stable_to_weak_graphicity}
There exists a connected finite \'etale cover $X'\to X$ such that the partition of $S|_{X'}$ into weakly graphic classes is finer than the pullback partition $\mathcal{S}|_{X'}$.
\end{lemma}

\begin{proof}
Let $R$ be the image of $S$ under the reduction map, it is a finite set since $S$ is relatively compact.
Choose a connected finite \'etale cover $X'\to X$ with sturdy reduction satisfying properties $(i)-(iv)$ from Lemma~\ref{s6:lem:special_cover_X'}, we will show that this cover satisfies the statement of the lemma.

Let $s_1,s_2 \in S$ be two points lying in different stable classes, choose two points 
\[
s'_1,s'_2\in X'(k)
\]
lying above $s_1$ and $s_2$, respectively.
For $i=1,2$ write $r'_i\in \mathcal{X}'(\kappa)$ for the reduction of $s'_i$ and denote
\[
X'_i = X'\setminus \{s'_i\}.
\]
We need to show that the reduction graphs
$\Gamma(X'_1)$ and $\Gamma(X'_2)$ are not isomorphic.
In the proof we will refer to properties $(i) - (iv)$ from Lemma~\ref{s6:lem:special_cover_X'} which hold for the cover $X'\to X$.

Note that hyperbolic curves $X'_i$ have only one cusp hence any isomorphism between their reduction graphs must preserve the vertex containing the cusp.
By properties $(i)$ and $(ii)$ there are vertical components $v'_i\in V(X'/X)$ containing $r'_i$, for $i = 1,2$.
As $s_1$ and $s_2$ lie in different stable classes the components $v'_1$ and $v'_2$ lie in different fibres of $X'\to X$, thus they have different genera by property $(iii)$.
Therefore if both points $r'_i$ lie in the smooth locus of the special fibre we see that the reduction graphs $\Gamma(X'_1)$ and $\Gamma(X'_2)$ are not isomorphic.
The same holds if only one point $r'_i$ lies in the smooth locus.
Indeed, if $r'_1$ is a node and $r'_2$ lies in the smooth locus then the cusp of $X'_1$ determines a component of genus zero, whereas on $X'_2$ the cusp determines a component having positive genus.

Thus we may suppose that both points $r'_i$ are nodes.
For $i=1,2$ denote by $w'_i\in V(X')$ the irreducible component different from $v'_i$ containing the node $r'_i$.
Then in the reduction of graph of $X'_i$ the cusp belongs to a vertex $u'_i$ of genus zero and the vertex $u'_i$ is adjacent to only two vertices $w'_i$ and $v'_i$.
However properties $(iii)$ and $(iv)$ imply that 
\[
g(v'_1)\notin\{g(w'_2), g(v'_2)\},
\]
thus again the reduction graphs are not isomorphic.
\end{proof}

As a corollary of the theory developed so far we can extend our results concerning families of curves separating classes of points.

\begin{proposition}\label{s6:prop:separating_compact_set_stable_classes}
Let $S\subset X(k)$ be a relatively compact subset and $\mathcal{S}$ be a finite partition of $S$ into stable classes.
Then there exists a connected finite \'etale cover $X'\to X$ and a finite set of families of hyperbolic curves $\{Y_j\}_{j\in J}$ over $X'$ which separate classes of $\mathcal{S}|_{X'}$ by the toric rank.
\end{proposition}

\begin{proof}
Using Lemma~\ref{s6:lem:stable_to_weak_graphicity} we find a connected finite \'etale cover $X'\to X$ such that the partition of $S|_{X'}$ into weakly graphic classes is finer than the restricted partition $\mathcal{S}|_{X'}$.
Thus by Lemma~\ref{s5:lem:finer_partition} we may replace $X$ by $X'$ and assume that $\mathcal{S}$ is a finite partition into weakly graphic classes.
Then the statement follows from Proposition~\ref{s5:prop:separating_compact_set}.
\end{proof}

We may improve Proposition~\ref{s6:prop:separating_compact_set_stable_classes} further by introducing finer partitions than the partition into stable classes.
Let $\mathcal{X}^{s}$ be a semistable model of $X$, similarly to the case of the stable model we make the following definition.

\begin{definition}
A partition of $X(k)$ obtained by taking the pullback of the trivial partition of $\mathcal{X}^{s}(\kappa)$ along the reduction map
\[
X(k)\to \mathcal{X}^{s}(\kappa)
\]
will be called a partition into \emph{semistable} classes.
\end{definition}

Clearly such a partition depends on the chosen semistable model $\mathcal{X}^s$, moreover any partition of $X(k)$ into semistable classes is finer that the partition into stable classes.
To relate these two partitions we will need the following strengthening of Theorem~\ref{s6:thm:tamagawa_rns}.

\begin{theorem}[Mochizuki-Tsujimura, {\cite[Theorem A]{mochizuki_tsujimura_2023}}]\label{s6:thm:rns}
There exists a connected finite \'etale cover $X'\to X$ which extends to a morphism $\mathcal{X}'\to \mathcal{X}^s$. 
\end{theorem}

Using Theorem~\ref{s6:thm:rns} we can now state our main result on separating finite partitions.

\begin{theorem}\label{s6:thm:separating_compact_set_semistable_classes} 
Let $S\subset X(k)$ be a relatively compact subset and $\mathcal{S}$ be a finite partition of $S$ into semistable classes.
Then there exists a connected finite \'etale cover $X'\to X$ and a finite set of families of hyperbolic curves $\{Y_j\}_{j\in J}$ over $X'$ separating classes of $\mathcal{S}|_{X'}$ by the toric rank.
\end{theorem}

\begin{proof}
Let $\mathcal{X}^s$ be a semistable model of $X$ giving rise to the partition $\mathcal{S}$.
By Theorem~\ref{s6:thm:rns} there exists a connected finite \'etale cover $X'\to X$ which extends to a morphism $\mathcal{X}'\to \mathcal{X}^s$.
Then the partition of $S|_{X'}$ into stable classes is finer than the restricted partition $\mathcal{S}|_{X'}$.
Thus by Lemma~\ref{s5:lem:finer_partition} we may replace $X$ by $X'$ and assume that $\mathcal{S}$ is the partition of $S$ into stable classes.
Therefore the result follows from Proposition~\ref{s6:prop:separating_compact_set_stable_classes}.\end{proof}

Using Theorem~\ref{s6:thm:separating_compact_set_semistable_classes} the proof of Theorem~\ref{s1:thm:separation_by_toric_rank} in now immediate.

\begin{proof}[Proof of Theorem~\ref{s1:thm:separation_by_toric_rank}]
Passing to a finite \'etale cover and taking the compactification we may assume that $X$ is a proper hyperbolic curve.
Given a finite subset $S\subset X(k)$ we can find a semistable model of $X$ such that the corresponding partition of $S$ into semistable classes is the trivial partition.
Thus we conclude by Theorem~\ref{s6:thm:separating_compact_set_semistable_classes}.
\end{proof}

From now on we proceed to proving Theorem~\ref{s1:thm:separation_in_char_zero} and we change our notation.
We write $k$ for an algebraically closed field of characteristic zero and $X$ for a hyperbolic curve over $k$; the following proof uses a well known argument.

\begin{proof}[Proof of Theorem~\ref{s1:thm:separation_in_char_zero}]
Passing to a finite \'etale cover and taking the compactification we may assume that $X$ is a proper hyperbolic curve.
There exists a normal finitely generated $\Z$-algebra $A\subset k$ and a family of hyperbolic curves $X_A$ over $\Spec A$ which is isomorphic to $X$ after a base change along $A\hookrightarrow k$.
We may also assume that $k$ is an algebraic closure of the fraction field of $A$.

By~\cite{cassels_1976} there exists a prime number $p$ and a $p$-adic local field $K$ together with an embedding $A\hookrightarrow \Oc_K$.
For an algebraic closure $\bar{K}$ of $K$ we can find an arrow $\iota\colon k\hookrightarrow \bar{K}$ making the following diagram commutative
\[
\begin{tikzcd}
\Oc_K \arrow[r, hook] & \bar{K}  \\
A\arrow[r, hook]\arrow[u, hook] &  k\arrow[u, hook, "\iota"'].
\end{tikzcd}
\]
Write $X_{\iota}$ for the base change of $X$ along $\iota$ and 
\[
S_{\iota}\subset X_{\iota}(\bar{K})
\]
for the set corresponding to $S$.
By construction $X_{\iota}$ has good reduction and the partition of $S_{\iota}$ into stable classes is the trivial partition, moreover we have an equivalence of categories
\[
\Fet (X_{\iota})\cong \Fet (X)
\]
of finite \'etale covers of $X_{\iota}$ and $X$.
Thus we can apply Lemma~\ref{s6:lem:stable_to_weak_graphicity} together with Proposition~\ref{s5:prop:separating_finite_set} to conclude.
\end{proof}

\section{Application to Galois sections}\label{s:galois_sections}

In the final section we will give an application of the theory developed in this paper to a problem concerning Galois sections over number fields, our main result is Theorem~\ref{s7:thm:main_semistable}.
We start by slightly generalizing our discussion from Section~\ref{s:introduction} and formulating Theorem~\ref{s1:thm:reduction_on_strongly_pos_density} also in the affine case.

Let $K$ be a field of characteristic zero with an algebraic closure $\bar{K}$, write $G_K = \Gal(\bar{K}/K)$ for the absolute Galois group of $K$.
For a hyperbolic curve $X$ over $K$ we have the homotopy short exact sequence 
\begin{equation}\label{s6:eq:etale}
1\to \pi_1(X\times_K \bar{K})\to \pi_1(X) \to G_K\to 1,
\end{equation}
denote by $\Sec(X)$ the set of conjugacy classes of sections of the above sequence.
There is a natural map
\[
X(K)\to \Sec(X) ,
\]
called the profinite Kummer map, which is injective when $K$ is a number field or a $p$-adic local field.
Additionally the set of $K$-rational cusps of $X$ determines a subset of sections
\[
\Cusp(X)\subset \Sec(X),
\]
called cuspidal sections, clearly if $X$ is proper then there are no cuspidal sections.

Assume now that $K$ is a number field.
Recall that a decomposition group $G_v\subset G_K$ of a valuation $v\in \V(K)$ may be identified with the absolute Galois group of $K_v$. 
Given a section $s$ of the sequence~\eqref{s6:eq:etale} and a valuation $v\in \V(K)$ we denote by $s_v$ the restriction of $s$ to a decomposition group $G_v\subset G_K$.
For a subset $\Omega\subset \V(K)$ we define
\[
res_{\Omega}\colon \Sec(X) \to \prod_{v\in \Omega} \Sec(X\times_K K_v)
\]
as the product of the restriction maps $s\mapsto s_v$ for all $v\in \Omega$.
The map $res_\Omega$ is known to be injective when $\Omega$ has density one (see~\cite{porowski2024}).
For a section $s\in \Sec(X)$ we say that $s$ is a \emph{Selmer} section if for every $v\in \V(K)$ the local section $s_v$ is either cuspidal or geometric, write 
\begin{equation}\label{s7:loc:sel_sec}
\Sel(X)\subset \Sec(X)
\end{equation}
for the subset of Selmer sections.
Clearly every cuspidal section is a Selmer section, moreover the profinite Kummer map induces an injection
\begin{equation}\label{s7:loc:X(K)_sel-cusp}
X(K)\hookrightarrow \Sel(X)\setminus \Cusp(X).
\end{equation}
The general form of the Section Conjecture predicts that both inclusions~\eqref{s7:loc:sel_sec} and~\eqref{s7:loc:X(K)_sel-cusp} are actually equalities.

Next we discuss reductions of Selmer sections, for a valuation $v\in \V(K)$ let $k_v$ be an algebraic closure of $K_v$ and $X_v$ be the base change of $X$ to $k_v$.
Write $\mathcal{X}_v$ for the stable model of $X_v$ over the ring of integers of $k_v$ and $\overline{\mathcal{X}}_v$ for its compactification.
Then we have the reduction map
\begin{equation}\label{s7:eq:reduction_map}
\overline{X}(K_v)\subset \overline{X}(k_v) \to \overline{\mathcal{X}}_v(\bar{\kappa}(v)),
\end{equation}
where $\bar{\kappa}(v)$ is the residue field of $k_v$; it is an algebraic closure of $\kappa(v)$.
For a Selmer section $s\in \Sel(X)$ and a valuation $v\in \V(X)$ the local section $s_v$ determines a unique point in $\overline{X}(K_v)$ which by abuse of notation we will denote by $s_v$, we also write $\bar{s}_v$ for the image of $s_v$ through the reduction map~\eqref{s7:eq:reduction_map}.
For a subset $\Omega\subset \V(K)$ we define the reduction map on $\Omega$
\[
red_{\Omega}\colon \Sel(X) \to \prod_{v\in \Omega} \overline{\mathcal{X}}_v(\bar{\kappa}(v))
\]
as the product of reductions $s\mapsto \bar{s}_v$ for all $v\in \Omega$. 
When $\Omega$ is infinite then the restriction of the map $red_{\Omega}$ to the subset $X(K)$ is injective.
On the other hand the restriction of $red_{\Omega}$ to the subset of noncuspidal Selmer sections is only known to be injective when $\Omega$ has density one (see \cite[Theorem~1.1]{porowski2025}).
Using our result from Sections~\ref{s:separating_classes} and~\ref{s:separating_points} we will be able to improve this result, as follows.

\begin{theorem}\label{s7:thm:strongly_pos_density}
Suppose that $\Omega$ has strongly positive density, then the restriction of the map $red_{\Omega}$ to the subset of noncuspidal Selmer sections is injective.
\end{theorem}

We recall that two Galois sections are conjugate if they are conjugate over an open subgroup of $G_K$ (see \cite[\S 3]{porowski2024}),
in particular to prove Theorem~\ref{s7:thm:strongly_pos_density} we may always replace $K$ by a finite field extension $L$ and $\Omega$ by its preimage in $\V(L)$; note that by definition the property of having strongly positive density is preserved under this operation.
Before we continue we show how to reduce the proof of Theorem~\ref{s7:thm:strongly_pos_density} to the case of a proper hyperbolic curve.

\begin{lemma}
Theorem~\ref{s1:thm:reduction_on_strongly_pos_density} implies Theorem~\ref{s7:thm:strongly_pos_density}.
\end{lemma}

\begin{proof}
Let $X$ be an affine hyperbolic curve and $s,t\in \Sel(X)$ be two noncuspidal sections such that 
\[
red_{\Omega}(s) = red_{\Omega}(t)
\]
for some set of valuation $\Omega\subset \V(K)$ of strictly positive density.
Fix a connected finite \'etale cover $X'\to X$ of genus at least two.
After passing to a finite field extension of $K$ we may assume that there exist lifts $s',t'\in \Sel(X')$ of $s$ and $t$ and a set of strongly positive density $\Omega'\subset \Omega$ such that 
\[
red_{\Omega'}(s') = red_{\Omega'}(t').
\]
If we show that $s'$ and $t'$ are conjugate then the same will hold for $s$ and $t$, hence replacing $X$ by $X'$ we may assume that $X$ has genus at least two; extending the base field $K$ we may also assume that $X$ has stable reduction over $K$.
Applying Theorem~\ref{s1:thm:reduction_on_strongly_pos_density} we deduce that the images of sections $s$ and $t$ along the natural map
\[
\Sec(X)\to \Sec(\overline{X})
\]
are equal.
Thus it follows from \cite[Lemma 3.6]{porowski2025} that $s$ and $t$ are conjugate.
\end{proof}

From now on we will assume that $X$ is a proper hyperbolic curve over $K$, in particular the set of cuspidal sections is empty.
For a subset $\Omega\subset \V(K)$ we say that $\Omega$ has \emph{potentially} density one if there exists a finite field extension $L/K$ such that the preimage of $\Omega$ to $\V(L)$ has density one.
For a valuation $v\in \V(K)$ we define a subset
\[
\Omega(v)\subset \V(K)
\]
consisting of all valuations $v'$ satisfying the following property: for any two Selmer sections $s,t\in \Sel(X)$ whenever we have the equality $\bar{s}_{v'} = \bar{t}_{v'}$ then also $\bar{s}_{v} = \bar{t}_{v}$.
With this definition Theorem~\ref{s1:thm:reduction_on_strongly_pos_density} will be obtained as a corollary of the following result. 

\begin{theorem}\label{s7:thm:main_stable}
For every $v\in \V(K)$ the set $\Omega(v)$ has potentially density one.
\end{theorem}

Let us first see how to finish the proof of Theorem~\ref{s1:thm:reduction_on_strongly_pos_density} using Theorem~\ref{s7:thm:main_stable}.

\begin{proof}[Proof of Theorem~\ref{s1:thm:reduction_on_strongly_pos_density}]
Let $\Omega\subset \V(K)$ be a set of strongly positive density and $s,t\in \Sel(X)$ be two Selmer sections such that 
\[
red_{\Omega}(s) = red_{\Omega}(t),
\]
we want to show that $s$ and $t$ are conjugate.
By Theorem~\ref{s7:thm:main_stable} for every valuation $v\in \V(K)$ the set $\Omega(v)$ has potentially density one, in particular the intersection 
\[
\Omega(v)\cap \Omega
\]
is nonempty.
Thus it follows from the definition of $\Omega(v)$ that $\bar{s}_{v} = \bar{t}_{v}$.
As this holds for all valuations $v\in \V(K)$ we deduce from \cite[Theorem 1.1]{porowski2025} that sections $s$ and $t$ are conjugate.
\end{proof}

We now proceed towards the proof of Theorem~\ref{s7:thm:main_stable}, it will be necessary to consider a slightly more general situation.
Let $S\subset \Sel(X)$ be a subset and $v\in \V(K)$ be a valuation.
For a finite partition $\mathcal{S}$ of $S$ we define a subset
\[
\Omega(v,S, \mathcal{S})\subset \V(K)
\]
consisting of all valuations $v'$ satisfying the following property: for any two sections $s,t \in S$ whenever we have the equality $\bar{s}_{v'} = \bar{t}_{v'}$ then $s$ and $t$ lie in the same class of the partition $\mathcal{S}$.
Recall from Section~\ref{s:separating_points} that we have defined various partitions of the set $X(k_v)$, hence using the natural map
\begin{equation}\label{s7:loc:reduction}
\Sel(X)\to X(K_v)\subset X(k_v) 
\end{equation}
we may pull back such a partition from $X(k_v)$ to $\Sel(X)$ to obtain a partition of $S$.

\begin{definition}
For $v\in \V(K)$ the classes of the partition of $\Sel(X)$ obtained as the pullback of a partition of $X(k_v)$ into stable (resp. weakly graphic, semistable) classes along the map~\eqref{s7:loc:reduction}
will be called $v$\emph{-stable} (resp. $v$\emph{-weakly graphic}, $v$\emph{-semistable}).
\end{definition}

Note that when $S = \Sel(X)$ and $\mathcal{S}$ is the partition of $S$ into $v$-stable classes then we have the equality $\Omega(v, S,\mathcal{S}) = \Omega(v)$, recovering our previous definition.
In particular, Theorem~\ref{s7:thm:main_stable} is a special case of the next theorem, which is our main result in this section.

\begin{theorem}\label{s7:thm:main_semistable}
Let $X$ be a proper hyperbolic curve over $K$ with and $v\in \V(K)$ be a valuation.
Let $S\subset \Sel(X)$ be a subset of Selmer sections and $\mathcal{S}$ be a finite partition of $S$ into $v$-semistable classes. 
Then the set $\Omega(v, S, \mathcal{S})$ has potentially density one.
\end{theorem}

In order to prove Theorem~\ref{s7:thm:main_semistable} we first need to analyse the behaviour of sets $\Omega(v,S,\mathcal{S})$ with respect to two operations: finite field extensions and passing to finite \'etale covers.
First, let $L/K$ be a finite field extension and let $w$ be a valuation of $L$ lying over $v$.
Then, considering $S$ as a subset
\[
S\subset \Sel(X\times_K L) ,
\]
it is easy to see that the set $\Omega(w,S,\mathcal{S})$ is equal to the preimage of $\Omega(v,S,\mathcal{S})$ under the restriction map $\V(L)\twoheadrightarrow \V(K)$.
In particular, to prove Theorem~\ref{s7:thm:main_semistable} we may replace $K$ by a finite field extension.

Next, let $X'\to X$ be a finite \'etale cover of hyperbolic curves over $K$.
After replacing $K$ by a finite field extension we may assume that the natural map
\[
\varphi\colon \Sel(X')\to \Sel(X)
\]
is surjective.
Write $S|_{X'}$ for the preimage $\varphi^{-1}(S)$ and $\mathcal{S}|_{X'}$ for the partition of $S|_{X}$ obtained by the pullback of the partition $\mathcal{S}$ via $\varphi$.
We claim that, up to finitely many exceptions, we have a containment
\begin{equation}\label{s7:loc:containment}
\Omega(v, S|_{X'},\mathcal{S}|_{X'})\subset \Omega(v,S,\mathcal{S}).
\end{equation}
Indeed, let $v'\in \V(K)$ be a valuation belonging to the left hand side of~\eqref{s7:loc:containment} and suppose that the residue characteristic of $v'$ is relatively prime to the degree of the Galois closure of $X'\to X$.
Let $s,t\in S$ be two Selmer sections such that 
\[
\bar{s}_{v'} = \bar{t}_{v'}.
\]
Then there exist sections $s',t'\in S|_{X'}$ lifting $s$ and $t$ such that 
\[
\bar{s}'_{v'} = \bar{t}'_{v'}.
\]
Since $v'$ belongs to the left hand side of~\eqref{s7:loc:containment} it follows that $s'$ and $t'$ lie in the same class of $\mathcal{S}|_{X'}$.
This implies that $s$ and $t$ lie in the same class of $\mathcal{S}$ hence $v'$ belongs to the right hand side of~\eqref{s7:loc:containment}, proving our claim.
In particular, to prove Theorem~\ref{s7:thm:main_semistable} we may replace $X$ by a finite \'etale cover and $\mathcal{S}$ by the pullback partition.

Finally, it is obvious that when $\mathcal{T}$ is another finite partition of $S$ which is finer than $\mathcal{S}$ then we have an inclusion
\[
\Omega(v,S,\mathcal{T})\subset \Omega(v, S,\mathcal{S}).
\]
Observe that to prove Theorem~\ref{s7:thm:main_semistable} we may combine the above properties together with Theorem~\ref{s6:thm:rns} and Lemma~\ref{s6:lem:stable_to_weak_graphicity}; thus we may replace $X$ by a finite \'etale cover, pass to a refinement of the pullback partition of $\mathcal{S}$ and assume that $\mathcal{S}$ is the partition of $S$ into $v$-weakly graphic classes.
Therefore we reduce the proof of Theorem~\ref{s7:thm:main_semistable} to proving the next proposition.

\begin{proposition}\label{s7:prop:weak_graphicity_partition}
Let $X$ be a proper hyperbolic curve over $K$ and $v\in \V(K)$ be a valuation.
Let $S\subset \Sel(X)$ be a subset of Selmer sections and $\mathcal{S}$ be a finite partition of $S$ into $v$-weakly graphic classes. 
Then the set $\Omega(v, S, \mathcal{S})$ has potentially density one.
\end{proposition}

Before we prove Proposition~\ref{s7:prop:weak_graphicity_partition} we need to recall a few facts from the theory of Galois representations.
Let $d\ge 1$ be a natural number, $\ell$ be a prime number and $B\subset \V(K)$ be a finite set of valuations containing all places above $\ell$.
Write $R = R(d, B, \ell)$ for the set of isomorphism classes of semisimple representations 
\[
\rho: G_K \to \GL(V)
\]
unramified outside $B$ and of half weight, where $V$ is a vector space over $\Q_{\ell}$ of dimension $\le d$.
Also, for a representation $\rho$ and a valuation $v\in \V(K)$ we denote by $\rho_v$ the restriction of $\rho$ to a decomposition subgroup $G_v\subset G_K$.

\begin{lemma}\label{s7:lem:finiteness}
The set $R$ is finite.
\end{lemma}

\begin{proof}
See for example~\cite[V, \S2]{faltings_wustholz_rational_points}.
\end{proof}

Note that whenever $d'\ge d$ and $B'\supset B$ then we naturally have an inclusion 
\[
R(d,B,\ell)\subset R(d',B',\ell).
\]
For two semisimple $G_K$-representations $\rho,\rho'$ we say that $\rho$ and $\rho'$ are \emph{almost isomorphic} if there exists an open subgroup $U\subset G_K$ such that their restrictions to $U$ are isomorphic.
Clearly this induces an equivalence relation on the set $R$, write 
\[
R' = R'(d, B,\ell)
\]
for the quotient of $R$ with respect to this equivalence relation.
Since $R$ is finite there exists a Galois finite field extension $L = L(R)$ of $K$ with the following property: 
two representations whose classes belong to $R$ are almost isomorphic if and only if their restrictions to $G_L\subset G_K$ are isomorphic.
Write 
\[
Split(L/K)\subset \V(K)
\]
for the set of all valuations $v\in \V(K)$ which are totally split in $L$.
For any semisimple representations $\rho,\rho'$ whose isomorphism classes lie in $R$ we define a subset 
\[
\Omega(\rho,\rho')\subset Split(L/K)
\]
consisting of all valuations $v\in Split(L/K)$ at which both representations $\rho,\rho'$ are unramified and the traces of Frobenius elements at $v$
\[
\tr(\rho(Frob_v)) \neq \tr(\rho'(Frob_v))
\]
are distinct.
Note that the definition of the set $\Omega(\rho,\rho')$ depends only on the classes of $\rho$ and $\rho'$ in $R'$.

\begin{lemma}\label{s7:lem:difference_set}
Assume that the images of $\rho$ and $\rho'$ in $R'$ are distinct. Then the set $\Omega(\rho,\rho')$ has potentially density one.
\end{lemma}

\begin{proof}
Suppose that $\Omega(\rho,\rho')$ does not have potentially density one, since $Split(L/K)$ has potentially density one it follows that the complement 
\[
Split(L/K)\setminus \Omega(\rho,\rho')
\]
has strongly positive density.
Then it follows from~\cite[Theorem 2]{rajan_1998} that $\rho$ and $\rho'$ are almost isomorphic.
\end{proof}

Finally we define
\[
\Omega(R) = \bigcap_{\rho\ncong\rho'} \Omega (\rho,\rho')
\]
with $\rho, \rho'$ running through all pairs of distinct isomorphism classes of representations from $R'$.
This is a finite intersection and from Lemma~\ref{s7:lem:difference_set} we deduce the next corollary.

\begin{corollary}\label{s7:cor:properties_of_Omega(R)}
The set $\Omega(R)$ has potentially density one and has the following property: if $\rho$ and $\rho'$ are two semisimple representations whose isomorphisms classes lie in $R$ and $v\in \Omega(R)$ is such that $\rho_v\cong \rho'_v$ then $\rho$ and $\rho'$ are almost isomorphic.
\end{corollary}

Next we recall a construction of Galois representations arising from a Galois section and a family of varieties. 
Let $Y\to X$ be a smooth proper family of curves of genus $g\ge 2$ over $X$, by considering the monodromy action of the fundamental group of $X$ on the first \'etale cohomology group 
\[
V = H^1_{et}(Y_{\bar{x}},\Q_{\ell})
\]
of a fibre over a geometric point $\bar{x}$ of $X$ we obtain a representation 
\[
\rho \colon \pi^{et}_1(X)\to \GL(V).
\]
Given a section $s\in \Sec(X)$ we may pull back $\rho$ along a homomorphism $s\colon G_K\to\pi^{et}_1(X)$ to obtain a representation
\[
\rho_s\colon G_K\to \GL(V).
\]
Moreover, there exists a finite set $B\subset \V(K)$ of valuations with the property that for every Selmer section $s\in \Sel(X)$ the representation $\rho_s$ is unramified outside $B$.
Write $\rho^{ss}_s$ for the semisimplification of the representation $\rho_s$ thus for $d=2g$ and $R = R(d,B,\ell)$ we obtain a map
\[
\phi_{Y} \colon \Sel(X)\to R
\]
sending a section $s$ to the isomorphism class of the representation $\rho^{ss}_s$.
After these preparations we finally come the the proof of Proposition~\ref{s7:prop:weak_graphicity_partition}.

\begin{proof}[Proof of Proposition~\ref{s7:prop:weak_graphicity_partition}]
Applying Proposition~\ref{s5:prop:separating_compact_set} we may replace $X$ by a finite \'etale cover and pass to a finite field extension of $K$ to assume that there exists a finite set of families $\{Y_j\}_{j\in J}$ of smooth proper curves over $X$ which satisfy the following condition:
\begin{quote}
$(*)$ for any two sections $s,t\in S$ lying in different classes of the partition $\mathcal{S}$ there exists an index $j\in J$ such that the Jacobians of the geometric fibres of $Y_j\to X$ over $s_v$ and $t_v$ have different toric rank.
\end{quote}
Pick a prime number $\ell$, then for sufficiently large integer $d$ and a finite set $B$ we obtain maps
\[
\phi_j = \phi_{Y_j}\colon \Sel(X) \to R,
\]
for every $j\in J$ and $R = R(\ell,d,B)$. We claim that we have an inclusion
\[
\Omega(R) \subset \Omega(v,S,\mathcal{S}),
\]
this will finish the proof as the set $\Omega(R)$ has potentially density one by Corollary~\ref{s7:cor:properties_of_Omega(R)}.

Indeed, let $v'\in \Omega(R)$ be a valuation and choose two Selmer sections $s,t\in S$ such that $\bar{s}_{v'} = \bar{t}_{v'}$.
This implies that for every $j\in J$ the local $G_{v'}$-representations 
\[
\phi_j(s)_{v'} \cong \phi_j(t)_{v'}
\]
are isomorphic.
By the definition of the set $\Omega(R)$ we deduce that for every $j\in J$ the global $G_K$-representations
$\phi_j(s)$ and $\phi_j(t)$ are almost isomorphic.

On the other hand, suppose that $s$ and $t$ lie in different classes of the partition $\mathcal{S}$.
Using property $(*)$ we may choose an index $j\in J$ such that the geometric fibres of $Y_j\to X$ over $s_v$ and $t_v$ have different toric ranks.
As the toric rank of a fibre can be reconstructed from the trace of a restriction of the corresponding local $G_v$-representation to an open subgroup we deduce that the representations $\phi_j(s)$ and $\phi_j(t)$ are not almost isomorphic which contradicts our previous conclusion.
Therefore $s$ and $t$ must lie in the same class of $\mathcal{S}$ which means that $v'\in \Omega(v,S,\mathcal{S})$.
\end{proof}

In the final part of this section we discuss some corollaries of Theorem~\ref{s7:thm:main_semistable}.
We keep the assumption that $X$ is a proper hyperbolic curve over a number field $K$.
For a valuation $v\in \V(K)$ we denote by $r(v)$ the cardinality of the image of the map
\[
\Sel(X)\to \mathcal{X}_v(\bar{\kappa}(v)). 
\]
Here is an immediate corollary of Theorem~\ref{s7:thm:main_stable}.

\begin{corollary}\label{s7:cor:bound_for_Omega(v)}
For all $v'\in \Omega(v)$ we have an inequality $r(v)\le r(v')$.
\end{corollary}

\begin{proof}
Put $n= r(v')$ and choose Selmer sections $s_i$ for $1\le i \le n$ whose reductions in $\mathcal{X}(\bar{\kappa}(v'))$ are pairwise different.
Then for every section $s\in \Sel(X)$ there is a unique index $i$ such that $\bar{s}_{v'} = \bar{s}_{i,v'}$.
From Theorem~\ref{s7:thm:main_stable} we deduce $\bar{s}_v = \bar{s}_{i,v}$ hence $r(v)\le n$
\end{proof}

Next we can also slightly strengthen~\cite[Lemma 6.2]{porowski2025}.

\begin{corollary}
Suppose that there exists an integer $d\ge 1$ and a set of valuations $\Omega\subset \V(K)$ of strictly positive density such that for every $v\in \Omega$ we have the inequality $r(v)\le d$.
Then the set $\Sel(X)$ is finite, of cardinality $\le d$.
\end{corollary}

\begin{proof}
Let $v\in \V(K)$ be a valuation, from Theorem~\ref{s7:thm:strongly_pos_density} there exists $v'\in \Omega\cap \Omega(v)$ thus from Corollary~\ref{s7:cor:bound_for_Omega(v)} we deduce that
\[
r(v)\le r(v') \le d.
\] 
Hence for all $v\in \V(K)$ we have $r(v)\le d$, thus it follows from \cite[Theorem 3.3]{porowski2025} and the proof of \cite[Lemma 6.2]{porowski2025} that the set $\Sel(X)$ is finite, of cardinality $\le d$.
\end{proof}

We end this section by considering some variants of the sets $\Omega(v)$.
For a valuation $v\in \V(K)$ write $\Oc_v\subset k_v$ for the valuation ring and fix an element $\pi_v\in \Oc_v$ of positive valuation.
Then for a natural number $n\ge 0$ we define a map $red_{v,n}$ as the composition
\[
red_{v,n}\colon \Sel(X)\to X(K_v) \to \mathcal{X}_v(\Oc_v/\pi_v^{n+1}),
\]
note that the image of $red_{v,n}$ is a finite set.
Let $S\subset \Sel(X)$ be a subset and define a set 
\[
\Omega(v, S, n)\subset \V(K)
\]
consisting of all valuations $v'\in \V(K)$ satisfying the following property: for any two Selmer sections $s,t\in S$ satisfying $\bar{s}_{v'} = \bar{t}_{v'}$ we also have $red_{v,n}(s) = red_{v,n}(t)$.
Note that there is a sequence of inclusions
\[
\Omega(v, S, 0)\supset \Omega(v, S, 1)\supset \ldots,
\]
moreover when $S = \Sel(X)$ then we have 
\[
\Omega(v)\supset \Omega(v, S, 0).
\]
Then our next corollary improves on Theorem~\ref{s7:thm:main_stable}.
\begin{corollary}\label{s7:cor:Omega(v,n)}
For a natural number $n\ge 0$ the set $\Omega(v, S, n)$ has potentially density one
\end{corollary}

\begin{proof}
Passing to a finite field extension we may assume that $X$ has stable reduction at $v$.
The map $red_{v,n}$ defines a finite partition $\mathcal{S}$ of the set $S$, as the pullback of the trivial partition of its codomain.
By considering semistable models of $X\times_K K_v$ arising from blow-ups at points in the special fibre of the stable model it is easy to see that there exists a partition $\mathcal{T}$ of $\Sel(X)$ into $v$-semistable classes which is finer than the partition $\mathcal{S}$.
Then we have
\[
\Omega(v, S, \mathcal{T})\subset \Omega(v, S,\mathcal{S}) = \Omega(v, S, n),
\]
thus the statement follows from Theorem~\ref{s7:thm:main_semistable}.
\end{proof}

For a valuation $v\in \V(K)$ we denote be $l(v)$ the cardinality of the image of the map
\[
\Sel(X) \to X(K_v).
\]
For a natural number $m\ge 0$ we define a subset
\[
\Omega(m)\subset \V(K)
\]
consisting all valuations $v\in \V(K)$ such that $l(v)\ge m$.
Thus we have a sequence of inclusions
\[
\V(K) = \Omega(0)\supset \Omega(1)\supset \ldots.
\]
Note that if $m\ge 1$ is such that $\Omega(m)$ has density zero then using \cite[Lemma 6.2]{porowski2025} we deduce that the set $\Sel(X)$ is finite, of cardinality $\le m-1$; in particular $\Omega(m)$ must be empty.
Moreover it follows from~\cite{lawrence_venkatesh_2020} (see also \cite{betts_stix_2025}) that for sufficiently large integers $m$ we have $\V(K)\neq \Omega(m)$,
according to the Section Conjecture this should imply that $\Omega(m)$ is empty.
Here we can prove a very weak variant of this conclusion.

\begin{corollary}
For a natural number $m\ge 0$ the set $\Omega(m)$ is either empty or has potentially density one.
\end{corollary}

\begin{proof}
Suppose that $\Omega(m)$ is nonempty and let $v\in \Omega(m)$.
Since $l(v)\ge m$ we see that there exists an integer $n\ge 0$ such that the image of the reduction map $red_{v,n}$ has cardinality $\ge m$.
Then using the same argument as in the proof Corollary~\ref{s7:cor:bound_for_Omega(v)} we deduce that for $S = \Sel(X)$ we have an inclusion
\[
\Omega(v, S, n)\subset\Omega(m),
\]
hence the result follows from Corollary~\ref{s7:cor:Omega(v,n)}.
\end{proof}

\bibliographystyle{abbrv}
\bibliography{bibliography}

\end{document}